\documentclass[11pt]{article}
\usepackage{amsmath, amssymb, amsthm}
\usepackage{enumitem}                                 
\usepackage[margin=0.8in]{geometry}

\AtEndDocument{
\bigskip{\footnotesize%
\textsc{Department of Mathematics, The University of Haifa at Oranim, Tivon 36006, Israel} \par 
  \textit{E-mail address} \texttt{ amir.algom@math.haifa.ac.il}
} 

\bigskip{\footnotesize%
\textsc{Department of Mathematics, School of Science, Wuhan University of Technology, Wuhan, 430070, P. R. China; and Department of Mathematical Sciences,  P.O. Box 3000,  90014  University of Oulu,  Finland} \par 
  \textit{E-mail address} \texttt{  chrischang2016@gmail.com
} 
}

\bigskip{\footnotesize%
\textsc{Department of Mathematical Sciences, P.O. Box 3000, 90014 University of Oulu, Finland} \par 
  \textit{E-mail address} \texttt{  meng.wu@oulu.fi
} 
}

\bigskip{\footnotesize%
\textsc{Department of Mathematical Sciences, P.O. Box 3000, 90014 University of Oulu, Finland} \par 
  \textit{E-mail address} \texttt{  yuliang.wu@oulu.fi
} 
}

}

\DeclareMathOperator{\supp}{supp}
\allowdisplaybreaks[1]

\newtheorem{theorem}{Theorem}[section]

\newtheorem{Counter-example}[theorem]{Counter example}

\newtheorem{Lemma}[theorem]{Lemma}
\newtheorem{Proposition}[theorem]{Proposition}

\newtheorem{Definition}[theorem]{Definition}
\newtheorem{Corollary}[theorem]{Corollary}

\newtheorem*{theorem*}{Theorem}

\newcommand{\ignore}[1]{}

\usepackage{xcolor}
\usepackage[pagebackref]{hyperref}
\hypersetup{
   colorlinks,
    linkcolor={red!60!black},
    citecolor={blue!60!black},
    urlcolor={blue!90!black}
}

\newcommand{\diam}{\text{diam}}

\allowdisplaybreaks[1]

\renewcommand\bigskip{\medskip}

\def\to{\rightarrow}

\def\R{\mathbb R}

\def\Z{\mathbb Z}

\DeclareMathOperator{\spt}{spt}

\DeclareMathOperator{\dist}{{\rm dist}}

\newtheorem{lemma}[theorem]{Lemma}

\theoremstyle{definition}

\newtheorem*{theorem-non}{Proposition}

\title{Van der Corput and metric theorems for geometric progressions for self-similar measures}

\author{Amir Algom, Yuanyang Chang,  Meng Wu, and Yu-Liang Wu}
\date{}

\begin{document}

\maketitle
\begin{abstract}
We prove a van der Corput lemma for non-atomic self-similar measures $\mu$. As an application,  we show that the correlations of all finite orders  of  $( x^n \mod 1 )_{n\geq 1}$ converge to the Poissonian model for $\mu$-a.e. $x$, assuming $x>1$. We also complete a recent result of Algom, Rodriguez Hertz, and Wang (obtained simultaneously by Baker and Banaji), showing that any self-conformal measure with respect to a non-affine real analytic IFS   has polynomial Fourier   decay. 
\end{abstract}

\section{Introduction} \label{Section intro}
\subsection{A van der Corput lemma for fractal measures} \label{Section intro 1}
A fundamental result in harmonic analysis is  the van der Corput Lemma: 
\begin{theorem-non} \cite[Proposition 2 in Chapter 8]{Stein1993book} \label{vanderCorputLemma}
 \textit{  Let $g$ be a real-valued smooth function on an interval $J \subset\mathbb{R}$. Suppose  $|g^{(k)}(x)|\geq 1$ for all $x\in J$, where $k\geq 1$ is an integer. If $k=1$ and $g'$ is monotonic, or  if $k>1$, then 
\begin{equation}\label{vanderCorput}
    \left|\int_{J} e^{2\pi i \lambda g(x)}dx\right|\leq c_k \lambda^{-\frac{1}{k}}\; \, \text{ for }\;\, \lambda>0, \text{ where } c_k>0 \text{  is independent of } \lambda.
\end{equation}}
\end{theorem-non}
The purpose of this paper is to prove variants of this result when the underlying measure in \eqref{vanderCorput} is a  self-similar measure, while assuming as little as possible about the smoothness of the phase function $g$. We will then apply this to study the fine statistics of sequences of the form $( x^n \mod 1 )_{n\geq 1}$ where $x$ is drawn according to such a measure, and to study the Fourier decay problem for self-conformal measures.

Self-similar measures are defined as follows:  Let $\Phi= \lbrace f_1,...,f_n \rbrace$ be a finite set of real non-singular contracting similarity maps of a compact interval $J\subset \mathbb{R}$. That is, for every $i$ we can write $f_i(x)=r_i \cdot x+t_i$ where $r_i \in (-1,1)\setminus \lbrace 0 \rbrace$ and $t_i \in \mathbb{R}$, and $f_i(J)\subseteq J$. We will refer to $\Phi$ as a \textit{self-similar IFS} (Iterated Function System).  It is well known that there exists a unique compact set $\emptyset \neq K=K_\Phi\subseteq J$ such that
$$ K = \bigcup_{i=1} ^n f_i (K).$$
The set $K$ is called a \textit{self-similar set}, and the \textit{attractor} of the IFS $\Phi$.  We always assume that there exist $i\neq j$ such that the fixed point of $f_i$ is not equal to the fixed point of $f_j$. This ensures that $K$ is infinite. 

 Next, let $\textbf{p}=(p_1,...,p_n)$ be a strictly positive probability vector, that is, $p_i >0$ for all $i$ and $\sum_i p_i =1$. It is well known that there exists a unique Borel probability  measure $\mu$ such that
\begin{equation} \label{Eq self-sim}
\mu = \sum_{i=1} ^n p_i\cdot  f_i\mu,\quad \text{ where } f_i \mu \text{ is the push-forward of } \mu \text{ via } f_i.
\end{equation}
The measure $\mu$ is called a \textit{self-similar measure}, and is supported on $K$.  Our assumptions that $K$ is  infinite and  that $p_i>0$ for every $i$ are known to imply that $\mu$ is non-atomic. In particular, all self-similar measures in this paper are  non-atomic. It is a  standard fact that the Lebesgue measure on any bounded interval  is a self-similar measure, but there are many self-similar measures that are not absolutely continuous.

We can now state one of the main results of this paper:

\begin{theorem} \label{thm:van_der_Corput_Ck}
    Let $\mu$ be a non-atomic self-similar measure supported on an interval $J \subseteq \mathbb{R}$.  Let $g \in C^k(J)$ for some integer $k\geq 2$ be such that for some $c_0>0$, $\max_{2\leq j\leq k}|g^{(j)}(x)| \geq c_0$ for any $x \in J$. Then  there exist constants $\tau=\tau(\mu,k)$ and $C=C(g)>0$ such that
    \begin{equation} \label{eq:van_der_Corput}
        \left| \int_{J} e^{2 \pi i g(x) \lambda} d \mu(x) \right| \le C  \lambda ^{-\tau}\quad  \text{for}\;  \lambda >0.
    \end{equation}
\end{theorem}
In fact, we will prove a more general result, where  the phase $g$ is only required to be $C^{1+\gamma}(\mathbb{R})$ and  non-flat in an appropriate sense; see Theorem \ref{thm:van_der_Corput_C1+a} in Section \ref{Section main tech} for more details.

The exponent $\tau$ as in Theorem \ref{thm:van_der_Corput_Ck} depends on the geometry of $\mu$, and is uniform across all $g$ satisfying its assumptions.  Some special cases of Theorem \ref{thm:van_der_Corput_Ck} had been established in the past: Kaufman \cite{Kaufman1984ber}  proved, among other things,  that if $g$ is any $C^2$ diffeomorphism on $[0,1]$ such that $g''>0$, then \eqref{eq:van_der_Corput} holds for the Cantor-Lesebgue measure on the middle-thirds Cantor set. This was later extended to all  self-similar measures with a uniform contraction ratio by Mosquera-Shmerkin \cite{Shmerkin2018mos}. Since the work of Mosquera-Shmerkin it remained an open problem as to whether the result holds for \textit{all} non-atomic self-similar measures (with no further assumptions).  Theorem \ref{thm:van_der_Corput_Ck} settles this problem, while also relaxing  the assumptions on the phase function. Finally,  Algom, Rodriguez Hertz, and Wang \cite{algom2021decay, algom2020decay} recently proved logarithmic  decay for $C^{1+\gamma}(\mathbb{R})$ phases under certain Diophantine assumptions on the IFS. No such assumptions are required for our corresponding result, Theorem \ref{thm:van_der_Corput_C1+a} below. See Section \ref{Section main tech} for more details.

We also remark that, simultaneously and independently of our work, Baker and Banaji \cite{Baker2023Banaji} obtained similar results but with a different method, and a somewhat different range of applications. Baker and Banaji rely on a disintegration technique for self-similar measures as in \cite{Algom2022Baker}, and apply  this to study Fourier decay for certain smooth images of fractal measures on the so-called fibered IFSs. Their technique applies in higher dimensions and to infinite IFSs. In contrast, we rely on an application of Tsujii's large deviations estimate \cite{Tsujii2015self} (see Section \ref{Section Sketch}), and apply this to study the fine statistics of sequences of the form $\left( x^n \mod 1 \right)_{n\geq 1}$ (Section \ref{Section corr}). However,  both results can be applied to show that any non-atomic self-conformal measure with respect to a non-affine real analytic IFS has polynomial Fourier decay, a result first stated in the recent work of Algom, Rodriguez Hertz, and Wang \cite[Corollary 1.2]{algom2023polynomial}. See Section \ref{Section decay} for more details.

\subsection{Applications to higher order correlations of sequences modulo 1} \label{Section corr}
A sequence $(x_n)_{n\geq 1}$ taking values in $[0,1)$ is called \emph{uniformly distributed} or \emph{equidistributed} if for any sub-interval $J\subseteq [0, 1)$, 
\begin{equation*}
    \lim_{N\rightarrow \infty}\frac{1}{N}\sharp\left\{1\leq n\leq N:\, 
    x_n \in J\right\}= |J|, \text{ where } |J|:=\text{ the Lebesgue measure of } J.
\end{equation*}
A  real valued sequence $(x_n)_{n\geq 1}$ is called \emph{uniformly distributed modulo one} if the sequence of fractional parts  $(\lbrace x_n \rbrace )_{n\geq 1}\subset [0,1)$ is uniformly distributed in $[0,1)$. This notion  concerns the proportion of the fractional parts of the sequence in a fixed interval. The \textit{fine-scale} statistics of sequences modulo one, which describe the behavior of a sequence on the scale of mean gap $1/N$, has attracted growing attention in recent years. Two of the most important fine-scale statistics are the $k$-level  correlations and the distribution of level spacings (also called nearest-neighbor gaps), which are defined as follows.

For every integer $N\geq 1$ and $k\geq 2$ let $\mathcal{U}_k=\mathcal{U}_k(N)$  denote the set of distinct integral $k$-tuples taking values in $\{1,\ldots,N\}$. That is,
$$\mathcal{U}_k=\left\{\textbf{u}=(u_1,\ldots,u_k):\, u_i\in\{1,\ldots,N\},\, u_i\neq u_j\, \text{ for all }  i\neq j\right\}.$$ 
For each $\textbf{u}\in\mathcal{U}_k$ and real valued sequence $(x_n)_{n\geq 1}$ consider the difference vector
$$\Delta(\textbf{u}, (x_n)_{n\geq 1})=(x_{u_1}-x_{u_2},\ldots,x_{u_{k-1}}-x_{u_k})\in\mathbb{R}^{k-1}.$$ 
Given $f: \mathbb{R}^{k-1}\rightarrow\mathbb{R}$ with compact support, the \emph{$k$-level correlation function} is defined to be
\begin{equation}\label{correlation}
    R_k(f, (x_n)_{n\geq 1}, N):=\frac{1}{N}\sum_{\textbf{u}\in\mathcal{U}_k}\sum_{\textbf{l}\in\mathbb{Z}^{k-1}}f(N(\Delta(\textbf{u}, \{x_n\}_{n\geq 1})+\textbf{l})).
\end{equation} 
If 
\begin{equation}\label{poisson}
    \lim_{N\rightarrow\infty} R_k(f, (x_n)_{n\geq 1}, N)=\int_{\mathbb{R}^{k-1}}f(\textbf{x})  d\textbf{x},\; \forall f\in C^{\infty}_c(\mathbb{R}^{k-1})
\end{equation}
then we say that \emph{the $k$-level correlation of $(\{x_n\})_{n\geq 1}$ is Poissonian}. This notion alludes to the fact that such behaviour is consistent with the almost sure behaviour of a Poisson process with intensity one.

We now define the distribution of level spacings or nearest-neighbour gaps of the sequence $(\{x_n\})_{n\geq 1}$, that is, the gaps between consecutive elements of $(\{x_n\})_{n\geq 1}$. For each $N\geq 1$, we reorder the sequence $\{x_n\}_{n=1}^N$ and label them as 
\begin{equation*}
   0\leq\theta_{1, N}\leq\theta_{2, N}\leq\dots\leq\theta_{N, N}\leq 1,\quad \text{ and set } \theta_{0, N}:=\theta_{N, N}-1 \pmod 1.
\end{equation*}
Suppose that for every $s\geq 0$ the limit as $N\rightarrow \infty$ of the function
\begin{equation*}
    G(s, \{x_n\}_{n\geq 1}, N):=\frac{1}{N}\sharp\left\{1\leq n\leq N:\, N\left( \theta_{n, N}-\theta_{n-1, N} \right)\leq s\right\}
\end{equation*}
exists. Then the limit function $G(s)$ is called the asymptotic distribution function of the level spacings of $(\{x_n\})_{n\geq 1}$. We say that  the level spacings are \emph{Poissonian} if $G(s)=1-e^{-s}$, which agrees with  the distribution of the waiting time of a Poisson process with intensity one. We refer the reader to \cite{Yesha2023Baker} for more exposition on why and how these notions are used to capture pseudo-randomness properties of sequences.

Recently, Aistleitner  and Baker \cite{Aistleitner2021Baker} proved that $(\{\alpha^n \})_{n\geq 1}$ has Poissonian pair correlations for almost every $\alpha>1$. This refined a classical Theorem of Koksma that such sequences are  equidistributed (almost surely), since it is well known that a sequence with Poissonian pair correlations is always equidistributed \cite{Chris2018Lach, Larcher2017equi, Jens2020equi}. Aistleitner  and Baker \cite{Aistleitner2021Baker} further conjectured that typically such sequences should have Poissonian $k$-level correlations. Very recently, Aistleitner, Baker, Technau, and Yesha \cite{Yesha2023Baker} proved this conjecture, and thus obtained that the level spacings of $\{\alpha^n\}_{n\geq 1}$ are Poissonian for a.e.~$\alpha>1$; here we use the well-known fact that if the  $k$-level correlation of a sequence is Poissonian for all $k\geq 2$, then the level spacings are also Poissonian \cite[Appendix A]{Rudnick1999par}.

The following theorem extends these results to all self-similar measures:
\begin{theorem}\label{Main Theorem pos}
Let $\mu$ be a non-atomic self-similar measure on $\mathbb{R}$ such that $\mu$-a.e.~$x$ is larger than $1$.  Let $\xi \in \mathbb{R} \setminus \lbrace 0 \rbrace$. Then for $\mu$-almost every $x$ the $k$-level correlation of $(\{\xi x^n \})_{n\geq 1}$ is Poissonian for all $k\geq 2$. In particular, the level spacings are also Poissonian, and the sequence $(\{x^n \})_{n\geq 1}$ is  uniformly distributed.
\end{theorem}
Baker \cite{bakwer2021equi} recently proved that for some self-similar measures, $(\{x^n \})_{n\geq 1}$ is  uniformly distributed almost surely. For example, this was shown for appropriate translates of the Cantor-Lesebgue measure on the middle-thirds Cantor set. Baker further conjectured \cite[Conjecture 1.1]{bakwer2021equi} that this should hold for all self-similar measures supported on $[1,\infty)$.  Theorem \ref{Main Theorem pos} proves this  conjecture. Also, as the Lebesgue measure restricted on any interval is a self-similar measure, Theorem \ref{Main Theorem pos}  generalizes the results of \cite{Yesha2023Baker} discussed above; moreover, our methods also yield an alternative proof of the Lebesgue case treated in that paper.

Finally, we remark that Theorem \ref{Main Theorem pos} follows from certain non-trivial and uniform estimates on the $\mu$-means of a family of exponential functions that arise naturally when considering the  $k$-level  correlations of $(\lbrace x^n \rbrace)_{n\geq 1}$ for $\mu$-typical $x$. See Theorem \ref{main-2} for more details.

\subsection{Applications to Fourier decay of self-conformal measures with respect to real analytic IFSs} \label{Section decay}
An important special case of Theorem \ref{thm:van_der_Corput_Ck} is when the phase $g$ is real analytic and not affine. By virtue of Lemma \ref{non-flat-analytic} proved below, Theorem  \ref{thm:van_der_Corput_Ck} applies to such phase functions, and we thus obtain the following notable corollary:
\begin{Corollary} \label{corollary-analyic-decay}
    Let $\mu$ be a non-atomic self-similar measure supported on an interval $ J\subseteq \mathbb{R}$, and let $g \in C^\omega (J)$ be a non-affine real analytic function.   Then there exist $\tau>0$ and $C>0$ such that
    \begin{equation} \label{eq:van_der_Corput analytic}
        \left| \int_{J} e^{2 \pi i g(x) \lambda} d \mu(x) \right| \leq C \left| \lambda \right| ^{-\tau}, \text{as}\; \left| \lambda \right|\rightarrow \infty.
    \end{equation}   
\end{Corollary}

Corollary \ref{corollary-analyic-decay} has important applications regarding the Fourier decay problem for self-conformal measures: Let $\Phi$ be a $C^\omega (J)$ IFS; This simply means that all the maps in $ \Phi$ are real analytic  strict contractions on the interval $J$. A great deal of attention has been given to the study of the rate at which the Fourier transform of  self-conformal measures with respect to such  IFSs (i.e., measures that satisfy \eqref{Eq self-sim} with respect to the IFS) decays. They arise naturally, e.g., as (finitely many) inverse branches of the Gauss map in number theory \cite{Sahl2016Jor} and as Furstenberg measures for some $\text{SL}(2,\mathbb{R})$ cocycles \cite{Yoccoz2004some, Avila2010jairo}, and are closely related to Patterson-Sullivan measures on limit sets of some Schottky groups \cite{Li2021Naud, Naud2005exp, Bour2017dya}, among others.

Combining Corollary \ref{corollary-analyic-decay} with  very recent results of Algom, Rodriguez Hertz, and Wang \cite{algom2023polynomial}, we can show that all such measures have polynomial Fourier decay, as long as $\Phi$ is not entirely made up of affine maps. Denoting by $\widehat{\nu} (\lambda)$ the Fourier transform of a Borel probability measure $\nu$, we have:
\begin{theorem} \label{Coro main}
Let $\Phi$ be a $C^\omega (\mathbb{R})$ IFS. If $\Phi$ contains a non-affine map then every non-atomic self-conformal measure $\nu$ admits some $\tau>0$ such that
\begin{equation*}
\widehat{\nu} (\lambda) = O\left(\frac{1}{|\lambda|^\tau} \right)\, \text{as}\; \left| \lambda \right|\rightarrow \infty.
\end{equation*}
\end{theorem}
This result was first announced in the work of Algom, Rodriguez Hertz, and Wang \cite[Corollary 1.2]{algom2023polynomial}. The last part of its proof in \cite[Section 6]{algom2023polynomial} required the estimate \eqref{eq:van_der_Corput analytic} that was announced in \cite[Theorem 6.5] {algom2023polynomial}. Thus, with \eqref{eq:van_der_Corput analytic} we complete the proof of Theorem \ref{Coro main}. Also, recall from our discussion at the end of Section \ref{Section intro 1} that Baker and Banaji \cite{Baker2023Banaji} obtained Theorem \ref{Coro main} simultaneously and independently of our work. See also the related work of Baker-Sahlsten \cite{baker2023spectral} regarding the non-linearizable case.

We emphasize that except for the existence of a non-affine map,  no further conditions are imposed on the IFS. Thus, Theorem \ref{Coro main} extends the recent work of Sahlsten-Stevens \cite[Theorem 1.1]{sahlsten2020fourier} by removing their total non-linearity and separation assumptions. Other related significant works include those of Li \cite{Li2018decay, li2018fourier} about Furstenberg measures, Bourgain-Dyatlov \cite{Bour2017dya} about Patterson-Sullivan measures, Baker and Sahlsten \cite{baker2023spectral} about totally non-linear $C^2$ IFSs, and our previous works \cite{algom2020decay, algom2021decay}. See the recent survey \cite{sahlsten2023fourier} for more discussion about the recent breakthroughs on the Fourier decay problem.

Finally, we remark that the only concrete examples of polynomial Fourier decay in the fully self-similar setting were given by Dai-Feng-Wang \cite{Dai2007Feng} and Streck  \cite{streck2023absolute} for some homogeneous IFSs. It is known, though, that most self-similar IFSs should satisfy this property; see Solomyak \cite{Solomyak2021ssdecay}. For more recent results  on Fourier decay for self-similar measures we refer to \cite{bremont2019rajchman, li2019trigonometric, Solomyak2021ssdecay, rapaport2021rajchman, varju2020fourier, algom2020decay, Shmerkin2018mos, Buf2014Sol, Dai2007Feng, Dai2012ber, streck2023absolute} and references therein.

\subsection{Main technical Theorem: van der Corput for phases with non-flat derivatives} \label{Section main tech}
In this Section we will state and discuss our main technical result, Theorem \ref{thm:van_der_Corput_C1+a}. It implies Theorem \ref{thm:van_der_Corput_Ck}, and thus Corollary \ref{corollary-analyic-decay}, as special cases. The  results regarding fine statistics of sequences,  Theorem \ref{Main Theorem pos}, follow from Theorem \ref{main-2}, which requires  delicate fine-tuning of the  proof of Theorem \ref{thm:van_der_Corput_C1+a} for certain exponential functions.

We first formulate the following definition. For a given bounded set $A\subseteq \mathbb{R}$ let $\mathcal{N}_r (A)$ denote the minimal number of balls of diameter $r$ needed to cover $A$.

\begin{Definition} \label{def:non-flat} 
    Let $0 < \delta \le 1$. A function $h: \R \to \R$ is called $\delta$-non-flat on an interval $J \subset \R$ if:

For all $\epsilon > 0$ there exists $n_0 = n_0(\epsilon)\geq 1$ such that for all $n > n_0$ and every bounded subset $A \subset J$  we have 
    \[ \mathcal{N}_{2^{-\delta n}}(A) \ge 2^{\epsilon n} \Rightarrow
        \text{diam}(h(A)) \ge 2^{-n}.
    \]
\end{Definition}
The idea is that the range of a $\delta$-non-flat function on a set of positive dimension can not be too small. Note that by definition, affine functions on $\mathbb{R}$ are 1-non-flat, and for any $0<\delta_1<\delta_2\leq 1$, $\delta_1$-non-flatness implies $\delta_2$-non-flatness. In addition, in  \cite[prior to Theorem 6.2]{Melo1993strein}, de Melo and van Strien discuss a related notion of non-flatness of a $C^1$ function: By their definition, a function $h\in C^1(J)$  is non-flat at a point $c\in J$ if there is a $C^2 (\mathbb{R},J)$ diffeomorphism $\varphi$ such that $\varphi(0)=c$ and $h\circ \varphi$ is a polynomial (that we assume is non-trivial). It follows from the results in Section \ref{Section flatness} that, locally, if $h$ is non-flat in the sense of de Melo and van Strien then it is non-flat in the sense of Definition \ref{def:non-flat}.

We can now state our main technical result.
\begin{theorem} \label{thm:van_der_Corput_C1+a}
   Let $\mu$ be a non-atomic self-similar measure supported on a closed interval $J\subseteq \R$. Let $g \in C^{1+\alpha}(J)$ for some $0 < \alpha \le 1$. Suppose that $g'$ is $\delta$-non-flat on $J$. Then there exist constants $\tau=\tau(\mu,\delta, \alpha)$ and $C=C(g) > 0$ such that 
        \[
        \left| \int e^{2 \pi i g(x) \lambda} d \mu(x) \right| \le C  \lambda  ^{-\tau }\quad  \text{for}\; \lambda>0.
        \]
\end{theorem}
By Lemma \ref{non-flat-smooth} if $g \in C^k(J)$ with $k \ge 1$ and  there is some $c_0>0$ with $\max_{1\leq j\leq k}|g^{(j)}(x)| \geq c_0$ for any $x \in J$, then it is $\frac{1}{2k}$-non-flat. Thus, Theorem \ref{thm:van_der_Corput_C1+a} implies Theorem  \ref{thm:van_der_Corput_Ck} as a special case.

\subsection{Outline of method} \label{Section Sketch}
In this section we sketch the proof of Theorem \ref{thm:van_der_Corput_Ck}. We work under the  stronger assumption that the phase $g$ satisfies $g''>0$ on the interval; the general case, as well as our other main results,  is proved via subtle optimizations  of this method. So, let $\mu$ be a non-atomic self-similar measure supported on, say, $[0,1]$, and let $g\in C^2 ([0,1])$ be a  diffeomorphism   such that $g''>0$. We show that there exists some $\alpha>0$ such that for $\lambda>0$
\begin{equation*}
\left| \int_{0} ^1 e^{2 \pi i g(x) \lambda} d \mu(x) \right| =  O\left( \lambda ^{-\alpha}\right).
\end{equation*} 

First, we recall the main ingredients in the approach of Kaufman \cite{Kaufman1984ber} and Mosquera-Shmerkin \cite{Shmerkin2018mos}, in the special case where the IFS is homogeneous; that is, all the maps in $\Phi = \lbrace f_1,...,f_n \rbrace$ share the same contraction ratio. Here, the measure $\mu$ is known to have a convolution structure, so  we may write explicitly
$$\widehat{\mu}(\lambda) = \prod_{k=1} ^\infty \left( \sum_{i=1} ^n p_i \exp\left(2 \pi t_i r^k \lambda \right) \right) , \text{ where } f_i(x)=r\cdot x+t_i \text{ for each }1\leq i \leq n.$$ 
From this decomposition, Kaufman, and later Mosquera-Shmerkin, first  showed  polynomial decay outside of a sparse set of frequencies:
\begin{equation} \label{Erdos-Kahane}
\forall \epsilon>0 \, \exists c_0>0\, \text{ s.t. for } T\gg 1, \, \left \lbrace \lambda\in [-T,T]:\, \left| \widehat{\mu}(\lambda) \right| \geq T^{-c_0}  \right \rbrace \text{ can be covered by } T^\epsilon \text{ intervals of length } 1.
\end{equation}
 In fact, Mosquera-Shmerkin \cite[Proposition 2.3]{Shmerkin2018mos} proved a quantitative version of this result, which enabled them to give explicit lower bounds on the Fourier dimension of $g\mu$ in some cases. The proof then proceeded by combining the previous two displayed equations; see, e.g., \cite[Theorem 3.1]{Shmerkin2018mos}. Tsujii \cite{Tsujii2015self}  extended \eqref{Erdos-Kahane} to all self-similar measures (without relying on convolution structure).  Recently, Khalil \cite[Corollary 1.8]{khalil2023exponential} proved a very general form of this result.

The main innovation of this paper is the introduction of a new method to  obtain polynomial Fourier decay for $g\mu$  from a large deviations estimate as in \eqref{Erdos-Kahane}, for $\mu$  \textit{without} a convolution structure, and for more general phase functions. Our initial observation is that we can rewrite \eqref{Erdos-Kahane}  as follows:
\begin{equation} \label{Erdos Kahane 2}
\forall \epsilon>0 \, \exists c>0\, \text{ s.t. for } t\gg 1,\quad \sharp \left\lbrace n\in \mathbb{Z}: \exists \lambda\in [n,n+1]\cap [-e^{t}, e^{t}] \text{ such that } \left| \widehat{\mu}(\lambda) \right|\geq e^{-c\cdot t} \right\rbrace  \leq e^{\epsilon \cdot t}.
\end{equation}

At this point we need some notation: Let us fix our self-similar IFS $\Phi= \lbrace f_1,...,f_n \rbrace$ and denote $\mathcal{A}:=\lbrace1,..,n\rbrace$. For every $i\in \mathcal{A}$ we write
$$f_i(x)=r_i \cdot x +t_i, \text{ where } 0<|r_i|<1  \text { and } t_i \in \mathbb{R}.$$ 
For every $\omega \in \mathcal{A}^* := \cup_{i \ge 1} \mathcal{A}^i$ of length $|\omega|=\ell$ write, recalling that $\mu=\mu_\mathbf{p}$ for the probability vector $\mathbf{p}=(p_1,...,p_n)$,
\begin{equation} \label{r of omega}
r_\omega:= r_{\omega_1}\cdot \cdot \cdot r_{\omega_\ell}, \text { and } f_\omega (x) =r_\omega\cdot x + t_\omega, \text{ where } t_\omega := f_\omega (0), \text{ and also } p_\omega = p_{\omega_1}\cdot \cdot \cdot p_{\omega_\ell}.
\end{equation}
For a word $\omega\in \mathcal{A}^*$ let $\omega^{-}$ denote the word of length $|\omega|-1$ obtained from $\omega$ by forgetting its last letter. For sufficiently small (not necessarily integer) $b>0$ we consider the following cut-set  of $\mathcal{A}^\mathbb{N}$:

\begin{equation} \label{Def W k}
\mathcal{W}_b  = \left\lbrace \omega \in \mathcal{A}^*: |r_\omega| \leq b < |r_{\omega^{-}}| \right \rbrace.
\end{equation}

Now, the first step in our proof is, using self-similarity and the $C^2$ assumption, to linearize: 
$$\left|\int_{0} ^1 e^{2\pi i \lambda g(x)}dx\right| = \left| \widehat{ g \mu} \left( \lambda \right) \right| \leq  \sum_{\omega \in \mathcal{W}_b} p_\omega\cdot \left| \widehat{\mu} \left( g'(t_\omega) r_\omega \lambda \right) \right|   + O\left(\lambda\cdot b^2 \right),$$
where  $b=b(\lambda)>0$ is an auxiliary parameter to be chosen later.

Next, fixing yet another auxiliary parameter $\epsilon>0$,  there is a $c_0>0$ such that \eqref{Erdos Kahane 2} holds when taking $t\approx \log \left( \lambda\cdot b \right)$. Thus, we can partition $\mathcal{W}_b $ into 
$$X:= \left \lbrace \omega:\,  \left| \widehat{\mu} \left( g'(t_\omega) r_\omega \lambda \right) \right| \geq \left| \lambda \cdot r_{\omega} \right|^{-c_0} \right\rbrace, \text{ and } \mathcal{W}_b \setminus X.$$ 
It is clear that 
$$\sum_{\omega \in \mathcal{W}_b \setminus X} p_\omega\cdot \left| \widehat{\mu} \left( g'(t_\omega) r_\omega \lambda \right) \right| < \left| \lambda \cdot b \right|^{-c_0}.$$
So, it remains to control 
$$\sum_{\omega \in \mathcal{W}_b} p_\omega\cdot \left| \widehat{\mu} \left( g'(t_\omega) r_\omega \lambda \right) \right|.$$

This is where \eqref{Erdos Kahane 2} and the non-linearity condition $g''>0$ enter the picture. Fix $n\in \mathbb{Z}$ such  that:
$$ [n,n+1] \text { contains some } g'(t_\omega) r_\omega \lambda  \text { with }  \left| \widehat{\mu} \left( g'(t_\omega) r_\omega \lambda \right) \right| \geq \left| \lambda \cdot r_{\omega} \right|^{-c_0}.$$ Roughly speaking, using the non-linearity of $g$ (or, more generally, the non-flatness of $g'$), and the H\"{o}lder regularity of $\mu$ (\cite{Feng2009Lau} and Proposition \ref{Lemma Feng}), one can show that for some uniform $s_0=s_0(\mu)>0$, 
\begin{equation} \label{Eq sum}
\sum_{\omega \in \mathcal{W}_b} p_\omega\cdot \textbf{1}_{\lambda \cdot g'(t_{\omega})\cdot r_{\omega}\in [n,n+1]} (\omega) \leq \mu \left( \text{ some ball of radius } \lambda \cdot b \right) =O\left( \left( \lambda \cdot b \right)^{-s_0} \right).
\end{equation}
The last displayed equation means that in each such interval $[n,n+1]$ the probability of having a "bad" frequency is very small; in addition, \eqref{Erdos Kahane 2} means that the number of such ``bad'' intervals is sub-polynomial - here it can be shown to be less than $\left( \lambda \cdot b \right)^{\epsilon}$. Thus, taking tally of our calculation and disregarding global multiplicative constants,
$$\left|\int_{0} ^1 e^{2\pi i \lambda g(x)}dx\right| \leq \left| \lambda \cdot b\right|^{-c_0}+\left| \lambda \cdot b \right|^{\epsilon}\cdot \left| \lambda \cdot b \right|^{-s_0}+\left| \lambda\cdot b \right|^2.$$
This  can be made polynomially small in $\lambda$ by a careful choice of the parameters $b$ and $\epsilon$ (note that $s_0=s_0(\mu)$ and $c_0=c_0(\epsilon)$ are not free for us to choose).

Finally, we remark that in practice we will derive the inequality \eqref{Eq sum} not on the full space $\mathcal{W}_b$, but rather only upon a further partition of this space into those $\omega,\eta$ such that $r_\omega=  r_\eta$. This introduces another factor, the number of such distinct $r_\omega$, that is large in $\lambda$ into the previously displayed equation; however, standard properties of self-similar IFSs show that it can be made logarithmic in $b=b(\lambda)$, and so the proof goes through.

\subsection{Acknowledgments}
We thank Federico Rodriguez Hertz, Zhiren Wang, Simon Baker,  Amlan Banaji,  Hongfei Cui,  and Osama Khalil, for many helpful comments. The first author is supported by Grant No. 2022034 from the United States - Israel Binational Science Foundation (BSF), Jerusalem, Israel. The last three authors are supported by the Academy of Finland,  project grant No. 318217.  The second author was partially supported
by NSFC 11901204 and 12271418.

\section{Preliminaries} \label{Section pre}
\subsection{Some properties of self-similar measures} 
We retain the notations setup in Section \ref{Section intro}, and in particular in Section \ref{Section Sketch}: Specifically, our self-similar IFS $\Phi= \lbrace f_1,...,f_n \rbrace$, $\mathcal{A}:=\lbrace1,..,n\rbrace$, and  the use of \eqref{r of omega} for symbolic notation. Recall that there is no loss of generality in assuming that 
$$K:=K_\Phi \subseteq [0,1].$$
This assumption will be retained throughout the paper.

Recall that we are  working with a self-similar measure $\mu=\mu_\mathbf{p}$ with respect to a strictly positive probability vector $\mathbf{p}\in \mathcal{P}(\mathcal{A})$. Since $K_\Phi$ is assumed to be infinite, this is well known to imply that $\mu$ is not atomic. In fact, as shown by Feng and Lau \cite{Feng2009Lau}, $\mu$ satisfies a much stronger H\"{o}lder regularity property: 
\begin{Proposition} \label{Lemma Feng} \cite[Proposition 2.2]{Feng2009Lau} Let $\mu$ be a non-atomic self-similar measure. Then there exists some $s_0=s_0(\mu)>0$ such that for all $r>0$
$$\sup_{x\in \mathbb{R}} \mu\left(B(x,r)\right) =O\left(r^{s_0}\right).$$
\end{Proposition}
Proposition \ref{Lemma Feng} means  that the Frostman exponent of $\mu$ is non-zero. We remark that it is quite common for this attribute of the measure to appear when estimating the decay of its Fourier transform, see e.g. \cite{Shmerkin2018mos, algom2020decay, algom2021decay, li2019trigonometric}; our estimates will involve it as well.

Next, recall the definition of the cut-set  $\mathcal{W}_b$ as in  \eqref{Def W k}. We denote the set of the corresponding contraction ratios by
\begin{equation} \label{Def ck}
\mathcal{C}_b := \left\lbrace r_{\omega}:\, \omega\in \mathcal{W}_b  \right\rbrace\subseteq (-1,1).
\end{equation}
We proceed to record a number of standard facts. Recall that $|\mathcal{A}|=n$. For $\omega \in \mathcal{A}^p$ we write $[\omega]$ to denote the set of  all $\eta \in \mathcal{A}^\mathbb{N}$ such that $\eta|_p = \omega|_p$. 
\begin{Lemma} \label{Lemma basic}
The following statements hold true for every $b>0$:
\begin{enumerate}
\item The collection of sets $\left\lbrace \left[\eta\right]:\, \eta\in \mathcal{W}_b  \right\rbrace$ forms a partition of $\mathcal{A}^\mathbb{N}$. Furthermore, 
$$\left| \mathcal{C}_b \right| = O \left( \log b \right), \text{ where the implicit constant is uniform in } b.$$

\item For any measurable map $g:\mathbb{R}\rightarrow \mathbb{R}$,
$$\int_{0} ^1 e^{2\pi i \lambda g(x)}\,d\mu(x) =  \widehat{ g \mu} \left( \lambda \right) =  \sum_{\omega \in \mathcal{W}_b} p_\omega\cdot \widehat{\mu} \left( g \left( f_\omega (x) \right) \lambda \right).$$
\end{enumerate}
\end{Lemma}
Indeed, Part (1) follows from a standard counting argument, whereas Part (2) follows from the stationarity of $\mu$ via Part (1).

Next, we discuss linearization. For any $\eta \in \mathcal{A}^*$ recall that we write 
$$f_\eta (x) =r_\eta\cdot x + t_\eta, \text { where } t_\eta := f_\eta (0).$$
We will require the following standard linear approximation lemma. Recall that we assuming $\mu$ is supported on $[0,1]$. 
\begin{Lemma} \label{Lemma lin}
Let $g\in C^{1+\alpha} ([0,1])$ where $\alpha>0$. Then for all $x \in [0,1]$ and $\eta \in \mathcal{A}^*$,
\begin{equation}
g\circ f_\eta (x) - g(t_\eta) - g'(t_\eta)\cdot r_\eta \cdot x = O_g\left( r_\eta ^{1+\alpha} \right).
\end{equation}
\end{Lemma}
\begin{proof}
There exists a point $y=y(x,\eta)\in [0,1]$ such that
$$g\circ f_\eta (x) = g(t_\eta) + g'(f_\eta(y))\cdot r_\eta \cdot x.$$
As $g' \in C^\alpha ([0,1])$ and $|x|,|y|\leq 1$, we thus have
$$\left| g\circ f_\eta (x) - g(t_\eta) - g'(t_\eta)\cdot r_\eta \cdot x  \right| = \left| g'  (f_\eta (y)) - g'(t_\eta)\right|\cdot |r_\eta|\cdot |x| \leq O_g\left( \left| f_\eta (y) - f_\eta (0) \right|^\alpha \right) \cdot |r_\eta| = O_g\left( r_\eta ^{1+\alpha} \right).$$
\end{proof}

\subsection{On Tsujii's large deviations estimate}
Recalling our discussion from Section \ref{Section Sketch}, let us first formulate Tsujii's original result from \cite{Tsujii2015self}. Writing $\mathcal{P}(X)$ for the family of Borel probability measures on the space $X$, and recalling that $\left| \cdot \right|$ stands for the Lebesgue measure on $\mathbb{R}$, we have:
\begin{theorem} \label{Tsuji original} \cite{Tsujii2015self}
Let $\mu \in \mathcal{P}(\mathbb{R})$ be a non-atomic self-similar measure. Then
\begin{equation*}
\lim_{c \rightarrow 0^+} \limsup_{t\rightarrow \infty} \frac{1}{t}\log  \left| \left\lbrace  \lambda\in  [-e^{t}, e^{t}]:\, \left| \widehat{\mu}(\lambda) \right|\geq e^{-c\cdot t} \right\rbrace \right|  =0.
\end{equation*}
\end{theorem}

In our argument we will require the following corollary of this result:
\begin{Corollary} \label{Tsuji} 
Let $\mu \in \mathcal{P}(\mathbb{R})$ be a non atomic self-similar measure. Then
\begin{equation*}
\lim_{c+0^+} \limsup_{t\rightarrow \infty} \frac{1}{t}\log \sharp \left\lbrace n\in \mathbb{Z}: \exists \lambda \in [n,n+1]\cap [-e^{t}, e^{t}] \text{ such that } \left| \widehat{\mu}(\lambda) \right|\geq e^{-c\cdot t} \right\rbrace  =0.
\end{equation*}
In particular, for every $\epsilon>0$ there exists $c_0>0$ such that for all $t$ large enough,
\begin{equation} \label{Eq. Tsuji}
\sharp \left\lbrace n\in \mathbb{Z}: \exists \lambda\in [n,n+1]\cap [-e^{t}, e^{t}] \text{ such that } \left| \widehat{\mu}(\lambda) \right|\geq e^{-c\cdot t} \right\rbrace  \leq e^{\epsilon \cdot t}.
\end{equation}
\end{Corollary}
\begin{proof}
Since $\mu$ is a self-similar measure, it is Borel and compactly supported. It is well known that the Fourier transform of such a measure is uniformly Lipschitz continuous \cite[Section 3.2]{Mattila2015Fourier}.

 Thus, if for some $c>0,t>0$ and $\lambda\in \mathbb{R}$ we have
$$\left|\widehat{\mu}(\lambda) \right| \geq  e^{-c\cdot t},$$
then, assuming $t$ is very large (depending on $c$),  for every $s\in [\lambda-e^{-2ct},\, \lambda+e^{-2ct}]$ we have
$$\left|\widehat{\mu}(s) \right| \geq \frac{1}{2} \cdot e^{-c\cdot t}.$$
This in turn implies that for every $c>0$ and $t=t(c)>0$ large enough,
$$e^{-2ct} \sharp \left\lbrace n\in \mathbb{Z}: \exists \lambda\in [n,n+1]\cap [-e^{t}, e^{t}] \text{ such that } \left| \widehat{\mu}(\lambda) \right|\geq  e^{-c\cdot t} \right\rbrace $$
$$ \leq \left| \left\lbrace  \lambda \in  [-e^{t}, e^{t}]:\, \left| \widehat{\mu}(\lambda) \right|\geq \frac{1}{2}\cdot e^{-c\cdot t} \right\rbrace \right|.$$
The first assertion of the corollary is a direct consequence of Theorem \ref{Tsuji original} and the last displayed equation.  As \eqref{Eq. Tsuji} is a formal consequence of this result, the corollary is proved.
\end{proof}

\subsection{Non-affine continuously differentiable functions are non-flat} \label{Section flatness}
The following two lemmas demonstrate the non-flatness of non-linear real continuously differentiable functions. 
\begin{lemma}\label{non-flat-smooth}
        Let $J\subset\mathbb{R}$ be a closed interval and $h \in C^k(J)$ for some integer $k\geq 1$. If 
        \begin{equation} \label{Eq assumption non flat}
                \max_{1\leq j\leq k}|h^{(j)}(x)| \ge 1 \text { for all } x \in J,
        \end{equation}
 then there exists $\rho>0$ such that $h$ is $\frac{1}{2k}$-non-flat on all sub-intervals $I \subseteq J$ with $|I|\leq \rho$.
\end{lemma}
We remark that in \eqref{Eq assumption non flat} we can replace $1$ with any $c_0>0$.
\begin{proof}
    Since $h \in C^k(J)$ there exists $\rho > 0$ such that 
    \begin{equation} \label{eq:uniform_continuity}
        \max_{0 \le j \le k}\sup \left\lbrace |h^{(j)}(x) - h^{(j)}(y)| : x,y \in J,\, |x-y| \le \rho \right\rbrace \le \frac{1}{4}
    \end{equation}
    Let $I \subseteq J$ be any sub-interval with $|I| \leq \rho$. We claim that $h$ is $\frac{1}{2 k}$-non-flat on $I$.

Suppose that $h$ is not $\frac{1}{2 k}$-non-flat on $I$. We will show that this contradicts our assumption \eqref{Eq assumption non flat}. Now, if $h$ is not $\frac{1}{2 k}$-non-flat on $I$,   then there exist $\epsilon >0$, arbitrarily large $n \geq 1$, and $A \subseteq I$ satisfying
$$\mathcal{N}_{2^{-\frac{n}{2k}}}(A) \ge 2^{n \epsilon}, \quad \text { and } \diam(h(A)) < 2^{-n}.$$
In particular, for $n$ large enough (depending on $\epsilon$), there exist $2^{k+1}$ points $\{x_i\}_{ i=1}^{2^{k+1}} \subseteq A$ such that
$$x_1<x_2<...<x_{2^{k+1}},\quad \text{ and } \dist(x_i,\,x_j) \ge 2^{-\frac{n}{2 k}} \text{ for all } i \ne j.$$
We proceed to  construct $2^{k}=2^{k+1}/2$ points $\{y^{(1)}_i\}_{i=0}^{2^{k}-1}$ as follows. By  the previous two displayed equations, for each $0 \le i \le 2^k-1$ there exists $y^{(1)}_{2i} \in [x_{2i+1}, x_{2i+2}]$ such that 
    \[
        |h'(y^{(1)}_{2i})| = \left| \frac{h(x_{2i+1}) - h(x_{2i+2})}{x_{2i+1} - x_{2i+2}} \right| \le 2^{-n} \cdot 2^{\frac{n}{2k}}.
    \]
   Note that the gaps between the (ordered) $2^k$ points $\{y^{(1)}_{2i}\}_{i=0}^{2^k-1}$ are at least $2^{-\frac{n}{2 k}}$. We next  construct $2^{k-1}=2^{k}/2$ points $\{y^{(2)}_i\}_{i=0}^{2^{k-1}-1}$ as follows:  For each $0 \le i \le 2^{k-1}-1$, there exists $y^{(2)}_{4i} \in [y^{(1)}_{4i}, y^{(1)}_{4i+2}]$ such that 
    \[
        |h^{(2)}(y^{(2)}_{4i})| = \left| \frac{h'(y^{(1)}_{4i+2}) - h'(y^{(1)}_{4i+4})}{y^{(1)}_{4i+2} - y^{(1)}_{4i+4}} \right| \le 2^{-n} \cdot 2 \cdot (2^{\frac{n}{2k}})^2.
    \]
    Note the use of the triangle inequality and the previous construction. 
    Continuing this process $k-2$ more times, we can
  find $k$ points $z_1, z_2, \cdots, z_k$ such that for every $1 \leq j \leq k$
    \[
        |h^{(j)}(z_j)| \le 2^{-n} \cdot 2^{k-1} \cdot \left(2^{\frac{n}{2k}}\right)^k \le \frac{1}{2}.
    \]
    In view of \eqref{eq:uniform_continuity}, it follows that 
    \[
    \max_{1 \le j \le k}\sup_{x \in I }|h^{(j)}(x)| \le \max_{1 \le j \le k} \left[|h^{(j)}(z_j)| + \sup_{x \in I }|h^{(j)}(x) - h^{(j)}(z_j)|\right] \le \frac{3}{4} < 1.
    \]
    This is a contradiction to our assumption \eqref{Eq assumption non flat}.  We thus conclude that  $h$ is not $\frac{1}{2 k}$-non-flat on $I$.
\end{proof}

The following lemma is a direct consequence of the previous one:
\begin{lemma}\label{non-flat-analytic}
        Let $g \in C^{\omega}(J)$ be a non-affine real analytic function on a compact interval $J\subseteq \mathbb{R}$. Then  there exist $c> 0$ and an integer $k\geq 2$ such that for all $x\in J$,
        \[
        \max_{2 \le j \le k} |g^{(k)}(x)| \ge c.
        \]
        In particular, there exists $\rho>0$ such that $g\cdot \frac{1}{c}$ is $\frac{1}{2k}$-non-flat on all sub-intervals $I \subseteq J$ with $|I|\leq \rho$.
\end{lemma}
\begin{proof}
    Suppose, towards a contradiction, that there exists a sequence of points $(x_n)_{n\geq 1}$ in $J$ such that $\max_{2 \le j \le n} |g^{(j)}(x_n)| \le \frac{1}{n}$ as $n\rightarrow \infty$. By compactness, it has an accumulation point $x \in J$. Now the smoothness of $g$ implies that $g^{(j)}(x) = 0$ for all $j \ge 2$. It thus follows that $g(y) = g(x) + g'(x) (y-x)$ on $J$, which is an affine function.
\end{proof}

\section{Proof of Theorem \ref{thm:van_der_Corput_C1+a}}
In this section we prove Theorem \ref{thm:van_der_Corput_C1+a}. 
Let $\mu=\mu_\mathbf{p}$ be our self-similar measure. We retain the assumption, without the loss of generality, that $\mu$ is supported on $[0,1]$, and  the other notations and definitions as in the  preceding sections, and specifically in Sections \ref{Section pre} and \ref{Section Sketch}. Let $g\in C^{1+\alpha}([0,1])$ be such that $g'$ is a $\delta$-non-flat function on $[0,1]$, and recall that  $0<\delta\leq 1$.

 Let $\lambda>0$. We aim to prove the existence of constants $\tau=\tau(\mu,\delta) > 0$ and $C=C(g)>0$  such that 
        \[
        \left| \int e^{2 \pi i g(x) \lambda} d \mu(x) \right| \le C  \lambda  ^{-\tau}.
        \] 
We will often work with big $O$-notation, and moreover sometimes even omit that for notational clarity, under the assumption that $\lambda\gg 1$. In particular, suitable constants $C$ can be  worked out from our argument.

We fix a parameter $b=b(\lambda)>0$ to be chosen later. The proof consists of 4 steps: 
$$ $$

\noindent{\textbf{Step 1: Linearization.}} 
Recall the definition of the cut-set \eqref{Def W k},
\begin{equation*} 
\mathcal{W}_b  = \lbrace \omega \in \mathcal{A}^*: |r_\omega| \leq b < |r_{\omega^{-}}| \rbrace.
\end{equation*}
Then we have:
\begin{eqnarray} 
\notag \left| \widehat{ g \mu} \left( \lambda \right) \right| &=& \left|  \sum_{\omega \in \mathcal{W}_b} p_\omega\cdot \widehat{\mu} \left( g \left( f_\omega (x) \right) \lambda \right) \right| \\
\notag &\leq & \left|  \sum_{\omega \in \mathcal{W}_b} p_\omega\cdot \left( \widehat{\mu} \left( g(t_\omega) + g'(t_\omega)\cdot r_\omega\cdot   \lambda \right)+O\left( r_\omega ^{1+\alpha}\cdot \lambda \right) \right)  \right| \\
\label{Eq first bound} &\leq&   \sum_{\omega \in \mathcal{W}_b} p_\omega\cdot \left| \widehat{\mu} \left( g'(t_\omega)\cdot r_\omega \cdot   \lambda \right) \right|+O\left( b ^{1+\alpha}\cdot \lambda \right).
\end{eqnarray}
Indeed, the first  equality is Lemma \ref{Lemma basic} Part (2), and the second inequality follows from Lemma \ref{Lemma lin} since $e^{2\pi i \lambda x}$ is a $2\pi \lambda$-Lipschitz function. The last one follows from the triangle inequality, and using that by definition $r_{\omega} = \Theta \left( b \right)$ for all $\omega \in \mathcal{W}_b $, where $\Theta \left(\cdot\right)$ is the standard big Theta notation.

$$ $$
\noindent{\textbf{Step 2: Applying Tsujii's theorem and the non-flatness assumption.}} Recall  definition  \eqref{Def ck}. Fix $r\in \mathcal{C}_b$ and define
\begin{equation} \label{Eq c k r}
\mathcal{W}_{b,r} = \lbrace \omega \in \mathcal{W}_b:\, r_{\omega} = r \rbrace.
\end{equation}
Recalling that we are free to choose $b=b(\lambda)$, we add the assumptions that as $b\rightarrow 0$, for all $\omega \in \mathcal{W}_{b}$,
$$|\lambda \cdot r_{\omega}| \rightarrow \infty \text{ and } |\lambda \cdot r _\omega ^{1+\alpha}| \rightarrow 0.$$
Note that this must hold uniformly in $\omega \in \mathcal{W}_b$ by definition.  In particular, for every $r\in \mathcal{C}_b$, we may assume $|\lambda\cdot r|$ is arbitrarily large.

Next, fix some $\epsilon>0$ to be determined later. Since $g\in C^{1+\alpha} ([0,1])$, there exists some $C>1$ such that
$$ \left|g'(x)\right|<C\, \text{ for all } x\in [0,1].$$ 
So, applying the Corollary \eqref{Eq. Tsuji} of Tsujii's theorem (Theorem \ref{Tsuji original}) with this $\epsilon$, there exists some $c_0=c_0(\epsilon)>0$ such that for all $\lambda$ large enough, putting $t=\log \left|\lambda\cdot r\cdot C\right|$ such that it grows to $\infty$ with $\lambda$, we have
\begin{equation} \label{Eq application of Tsuji}
\sharp \left\lbrace n\in \mathbb{Z}: \exists \omega \in \mathcal{W}_{b,r} \text{ with } \lambda\cdot g'(t_{\omega})\cdot r \in [n,n+1] \text{ such that } \left| \widehat{\mu} \left( \lambda\cdot g'(t_{\omega})\cdot r\right) \right|\geq \left| \lambda \cdot r\cdot C  \right| ^{-c_0} \right\rbrace | \leq \left| \lambda \cdot C \cdot r \right| ^\epsilon.
\end{equation}

The next Lemma is the only place in the proof where the $\delta$-non-flatness assumption is used. Recall that $s_0>0$ is the Frostman exponent of the measure as in Proposition \ref{Lemma Feng}. Also, to streamline notation, let $\mathbb{P} = \mathbf{p}^\mathbb{N}$ be the Bernoulli on $\mathcal{A}^\mathbb{N}$. Let $\pi:\mathcal{A}^\mathbb{N}\rightarrow K$ be the coding map $\pi(\eta)=\lim_{n\rightarrow \infty} f_{\eta|_n} (0)$, where $\eta|_n\in \mathcal{A}^n$ is the $n$-prefix of $\eta$. It is standard that $\pi \left(\mathcal{A}^\mathbb{N}\right)=K$ and that $\pi \mathbb{P}=\mu$. Finally, for $\omega \in \mathcal{A}^*$ recall that  $[\omega]=\lbrace \eta\in \mathcal{A}^\mathbb{N}: \, \eta|_{|\omega|} = \omega\rbrace$.
\begin{Lemma} \label{Lemma prob of bad}
For every fixed $n\in \mathbb{Z}$,
$$\mathbb{P}\left( [\omega]:\, \omega\in \mathcal{W}_{k,r},\lambda\cdot g'(t_{\omega})\cdot r \in [n,n+1] \right) = O\left( (r\cdot  \lambda)^{\frac{\delta\cdot s_0}{2}} \cdot \left( (r\cdot\lambda)^{-\delta}+2r \right)^{s_0} \right).$$
\end{Lemma}
\begin{proof}
Let 
$$A:= \lbrace t_\omega:\, \omega\in \mathcal{W}_{k,r},\, \lambda\cdot g'(t_{\omega})\cdot r \in [n,n+1]\rbrace = \left( g' \right)^{-1} \left(  \left[\frac{n}{r\cdot  \lambda},\, \frac{n+1}{r\cdot \lambda}\right] \right)\bigcap \left \lbrace t_\omega : \, \omega \in  \mathcal{W}_{b,r} \right \rbrace  .$$
Then $\diam \left( g'(A) \right) \leq \frac{1}{r\cdot \lambda}$. Using the exponent $\epsilon = \frac{s_0}{2}$, we deduce from the $\delta$-non-flatness of $g'$ that 
$$\mathcal{N}_{(r\cdot\lambda)^{-\delta}} \left( A \right) <(r\cdot  \lambda)^{\frac{\delta\cdot s_0}{2}}.$$
Note that here we are assuming $\lambda$ is sufficiently large. Recalling that $\pi:\mathcal{A}^\mathbb{N}\rightarrow K$ is the coding map, for every $\omega \in \mathcal{W}_b$ and $\eta \in [\omega]$ we have
$$\left| \pi(\eta) - t_\omega \right| = \left| \pi(\eta) - f_\omega(0) \right| \leq r.$$
So, if $t_\omega \in I$, where $I$ is an interval of length $(r\cdot\lambda)^{-\delta}$, then $f_\omega ([0,1])$ belongs to an interval of size $(r\cdot\lambda)^{-\delta}+2r$. Thus, combining with our previous estimates,
$$\mathcal{N}_{(r\cdot\lambda)^{-\delta}+2r} \left( \bigcup_{\omega\in \mathcal{W}_{k,r}:\,\, \lambda\cdot g'(t_{\omega})\cdot r \in [n,n+1]} f_\omega([0,1]) \right) <(r\cdot  \lambda)^{\frac{\delta\cdot s_0}{2}}.$$
Applying Proposition \ref{Lemma Feng} and Boole's inequality, we obtain
$$\mathbb{P}\left( [\omega]:\, \omega\in \mathcal{W}_{k,r},\lambda\cdot g'(t_{\omega})\cdot r \in [n,n+1] \right) = O\left( (r\cdot  \lambda)^{\frac{\delta\cdot s_0}{2}} \cdot \left( (r\cdot\lambda)^{-\delta}+2r \right)^{s_0} \right).$$
This proves the lemma.
\end{proof}

$$ $$

\noindent{\textbf{Step 3: Collecting the error terms.}} First, recalling \eqref{Eq c k r} and the terms $c_0,C>0$ from \eqref{Eq application of Tsuji},  we partition $\mathcal{W}_{k,r}$ into two subsets:
\begin{itemize}
\item $\mathcal{W}_{b,r,1}$ is the collection of all $\omega \in \mathcal{W}_{b,r} $ such that
$$ \left| \widehat{\mu} \left( g'(t_\omega)\cdot r_\omega \cdot \lambda   \right) \right|\geq \left| \lambda \cdot r\cdot C \right|^{-c_0}.$$

\item $\mathcal{W}_{b,r,2}$ is the collection of all the other $\omega \in \mathcal{W}_{k,r}$, that is,  
$$\mathcal{W}_{b,r,2} = \mathcal{W}_{b,r} \setminus \mathcal{W}_{b,r,1}.$$ 
\end{itemize}
Now, for every $r\in \mathcal{C}_{k}$, applying \eqref{Eq application of Tsuji} and  Lemma \ref{Lemma prob of bad} with the notation $\mathbb{P}$ as in that lemma yields\begin{eqnarray*}
\sum_{\omega \in \mathcal{W}_{b,r}} p_\omega\cdot \left| \widehat{\mu} \left( g'(t_\omega)\cdot r \cdot   \lambda \right) \right| &= & \sum_{\omega \in \mathcal{W}_{b,r,1}} p_\omega\cdot \left| \widehat{\mu} \left( g'(t_\omega)\cdot r \cdot   \lambda \right) \right|+ \sum_{\omega \in \mathcal{W}_{b,r,2}} p_\omega\cdot \left| \widehat{\mu} \left( g'(t_\omega)\cdot r \cdot   \lambda \right) \right| \\
&\leq & \sum_{n\in \mathbb{Z}} \mathbb{P} \left(  [\omega]:\, \omega \in \mathcal{W}_{b,r,1} \text{ with } \lambda\cdot t_{\omega}\cdot r\in [n,n+1] \right)+\\
& & \mathbb{P} \left(  [\omega]:\, \omega \in \mathcal{W}_{b,r,2} \right)\cdot \left| \lambda \cdot r\cdot C \right|^{-c_0} \\
&\leq &  \left|\lambda \cdot r\cdot C \right|^\epsilon \cdot O\left( (r\cdot  \lambda)^{\frac{\delta\cdot s_0}{2}} \cdot \left( (r\cdot\lambda)^{-\delta}+2r \right)^{s_0} \right)+ \left| \lambda \cdot r\cdot C \right|^{-c_0}.
\end{eqnarray*}

Combining this with \eqref{Eq first bound} and Lemma \ref{Lemma basic} Part (1), noting that $r = \Theta \left( b \right)$ and omitting big-$O$ notations, we obtain
\begin{eqnarray}
\notag \left| \widehat{ g \mu} \left( \lambda \right) \right|  &\leq&   \sum_{\omega \in \mathcal{W}_b} p_\omega\cdot \left| \widehat{\mu} \left( g'(t_\omega)\cdot r_\omega \cdot   \lambda \right) \right|+\left|b ^{1+\alpha}\cdot \lambda \right| \\
\notag  &=& \sum_{r \in \mathcal{C}_b} \sum_{\omega \in \mathcal{W}_{b,r}} p_\omega\cdot \left| \widehat{\mu} \left( g'(t_\omega)\cdot r \cdot   \lambda \right) \right| +\left|b ^{1+\alpha}\cdot \lambda \right| \\
\notag  &\leq & \left| \mathcal{C}_b  \right| \cdot \left( \left|\lambda \cdot b\cdot C \right|^\epsilon \cdot \left| (b\cdot  \lambda)^{\frac{\delta\cdot s_0}{2}} \cdot \left( (b\cdot\lambda)^{-\delta}+2b \right|^{s_0} \right) +  \left|\lambda \cdot b\cdot C \right|^{-c_0} \right) + \left|b ^{1+\alpha}\cdot \lambda \right| \\
\label{Eq final}&\leq &   \log (b)\ \cdot \left( \left|\lambda \cdot b\cdot C \right|^\epsilon \cdot \left| (b\cdot  \lambda)^{\frac{\delta\cdot s_0}{2}} \cdot \left( (b\cdot\lambda)^{-\delta}+2b \right)^{s_0} \right) +  \left|\lambda \cdot b\cdot C \right|^{-c_0} \right) +\left|b ^{1+\alpha}\cdot \lambda \right|. 
\end{eqnarray}
$$ $$

\noindent{\textbf{Step 4: Choice of parameters and conclusion of the proof.}} Recall that we were  given some $\lambda$, and that we may assume $\lambda \gg 1$. We are now in a position to determine $b$ and $\epsilon$; As for the other parameters in \eqref{Eq final}, note that $c_0 = c_0 (\epsilon)$, $s_0=s_0(\mu)$, and that $C=C(g)$ is a global constant (independent of $\lambda$). We also  recall that before Lemma \ref{Lemma prob of bad} we added the conditions that as $b\rightarrow 0$,  uniformly for $\omega \in \mathcal{W}_{b}$,
$$|\lambda \cdot r_{\omega}| \rightarrow \infty \text{ and } |\lambda \cdot r _\omega ^{1+\alpha}| \rightarrow 0,$$
which is the same as asking that, as $r_{\omega} = \Theta \left( b \right)$,
$$|b\cdot \lambda| \rightarrow \infty \text{ and } |b^{1+\alpha}\cdot \lambda| \rightarrow 0  \text{ as } \lambda\rightarrow \infty.$$

Recalling that $\alpha>0$ by our assumption, find any $\gamma$ such that
$$ \frac{1}{1+\alpha}<\gamma<1.$$
We pick 
$$b=\lambda^{ -\gamma},\quad \epsilon = \frac{\delta\cdot s_0}{4}.$$
Note that, as $0<\gamma,\delta,s_0\leq 1$,
$$\epsilon(1-\gamma)+\frac{\delta\cdot s_0}{2}(1-\gamma)-s_0\cdot \gamma\leq \frac{3s_0}{4}(1-\gamma)-s_0\cdot \gamma <1.$$
Thus,  omitting global multiplicative constants (in particular, the big-$O$ notation,  absolute values, and multiplicative constants), via \eqref{Eq final} we have
\begin{eqnarray*}
\notag \left| \widehat{ g \mu} \left( \lambda \right) \right| &\leq& \log (\lambda) \cdot \left( \left|\lambda \cdot b \right|^\epsilon \cdot \left( (b\cdot  \lambda)^{\frac{\delta\cdot s_0}{2}} \cdot \left( (b\cdot\lambda)^{-\delta\cdot s_0}+b^{s_0} \right) \right) +  \left|\lambda \cdot b \right|^{-c_0} \right) +\left|b ^{1+\alpha}\cdot \lambda \right| \\
&\leq &\log (\lambda) \cdot  \left( \lambda^{-(1-\gamma)\cdot \frac{s_0 \cdot \delta}{4} } +\lambda^{\epsilon(1-\gamma)+\frac{\delta\cdot s_0}{2}(1-\gamma)-s_0\cdot \gamma}   +  \lambda^{-c_0\cdot (1-\gamma)} \right)  +\lambda^{1-(1+\alpha)\cdot \gamma}. 
\end{eqnarray*}
This proves the existence of $\tau$ as in Theorem \ref{thm:van_der_Corput_Ck}. As a constant $C$ can be worked out from this argument by re-considering the implicit constants in the omitted big-$O$ notations, Theorem \ref{thm:van_der_Corput_Ck}  is proved.

\section{Proof of Theorem \ref{Main Theorem pos}}
In this section we prove Theorem \ref{Main Theorem pos}. The general idea is to prove a quantitative van der Corput type estimate for a family of polynomials that arise naturally when considering the $k$-level correlation sums \eqref{correlation}, and the Poisson summation formula. Once this theorem is established, we will derive Theorem \ref{Main Theorem pos} from it via standard arguments.

Let us state the main technical theorem of this section:
\begin{theorem}\label{main-2}
Let $\mu \in \mathcal{P}([a,b])$ be a non-atomic self-similar measure, where $1<a<b$. Let $k\geq 2$ be an integer, $N\gg 1$ and $0<\epsilon\ll 1$, and let  $g\in \mathbb{Z}[x]$ be a polynomial with integer coefficients of the form
\begin{equation} \label{polynomial}
    g(x)=\sum_{j=1}^{k-1}l_j (x^{u_j}-x^{u_{j+1}})+\sum_{j=1}^{k-1}m_j (x^{v_j}-x^{v_{j+1}}),\, 
\end{equation}
$\text{ where } |l_j|, |m_j|\in(0, N^{1+\epsilon}], 1\leq u_k<\ldots<u_1\leq N, 1\leq v_k<\ldots<v_1\leq N.$

Suppose that either $u_1>v_1$, or 
$u_1=v_1$ but $l_1+m_1\neq 0$. If $u_1=\deg(g)\geq N^{\frac{1}{2k}}$, then there exists $\tau>0$ such that
\begin{equation}\label{eq:main-thm-estimate-1}
\left|\int e^{2\pi i g(x) }d\mu(x) \right|\le a^{-N^{\tau}},
\end{equation}
where $\tau$ is independent of $g$.
\end{theorem}
The proof of Theorem \ref{main-2} will take up the majority of this section. We remark that its proof is based upon the scheme of proof of Theorem \ref{thm:van_der_Corput_C1+a}. Also,  a significant part of our argument is devoted to study of the  of the polynomials as in \eqref{polynomial}. This discussion takes place in the next section.
\subsection{A covering lemma for a class of polynomials} \label{Section quan}
Let us first define a family of polynomials that will play an essential role in the proof of Theorem \ref{main-2}:
\begin{Definition} \label{Def poly}
Let $k\geq 2$ be an integer, and let $N\geq 2$. Define
\begin{equation}
   \mathcal{F}_{k, N}:=\left\{\sum_{j=1}^k l_j x^{u_j}:\, l_j\in\mathbb{Z},\, l_j\neq 0,\, |l_j|\leq N^4, u_j\in\mathbb{Z},\, 0\leq u_k<\ldots<u_1\leq N\right\}\subseteq \mathbb{Z}[x].
\end{equation}
\end{Definition}
By Descartes' sign rule, each polynomial in $\mathcal{F}_{k, N}$ has at most $2k$ distinct roots in $\mathbb{R}$ \cite[Lemma 3]{Cucker1999Smale}. We have the following quantitative covering Lemma concerning  the elements  of $\mathcal{F}_{k, N}$. 

\begin{lemma}\label{smallvalue1}
    Fix an integer $k\geq 2$ and reals $1<a<b$. For any $\epsilon>0$ there exists $\sigma\in (0, 1)$ such that for all large enough $N$ and $h\in\mathcal{F}_{k, N}$, there exist at most $2k+2$ sub-intervals $I_i\subset [a, b]$ satisfying :
\begin{enumerate}
    \item $|I_i|\leq a^{-N^{\sigma}}$ for each $i$, and
    \item $|h(x)|\geq a^{u_1-N^{\epsilon}}$ for all $x\in [a, b]\setminus\bigcup_i I_i$.
\end{enumerate}
\end{lemma}

Our proof of Lemma \ref{smallvalue1} relies on the following lemma. 
\begin{lemma}\label{smallvalue2}
    Fix an integer $k\geq 2$ and reals $1<a<b.$ Let $m\geq 1$ and let $P(x) \in\mathbb{R}[x]$ be a  polynomial such that $\deg(P)=m$ and:
$$P(x)=\sum_{j=0}^m l_j x^j,\quad |l_m|\geq 1,\quad \sharp\{0\leq j\leq m: l_j\neq 0\}\leq k.$$

Then for any $q> 2(1+\log_a 2+\log_a 3)$ there exist at most $2k+2$ sub-intervals $I_i\subset [a, b]$ such that:
\begin{enumerate}
    \item $|I_i|\leq a^{-mq}$ for each $i$;
    \item $|P(x)|\geq a^{-(mq)^2}$ for all $x\in [a, b]\setminus\bigcup_i I_i$.
\end{enumerate}
\end{lemma}
\begin{proof}
    Let 
    $$A=\big\{x\in [a, b]: |P(x)|< a^{-(mq)^2}\big\}.$$
Since $P$ is a polynomial with less than $k$ terms, both $P$ and $P'$ have at most $2k$ real roots. Thus $A$ is the union of at most $2k+2$ intervals. Suppose that there exists an interval $I\subseteq [a, b]$ satisfying
\begin{equation}\label{eq:lemma-smallvalue2,1}
    |I|> a^{-mq}\;\text{and}\;|P(x)|< a^{-(mq)^2}\;\text{for all}\;x\in I.
\end{equation}

We divide the interval $I$ into $3^m$ consecutive sub-intervals of equal length: Initially, the interval $I$ is divided into $3$  sub-intervals of equal length,  $I_1,I_2,I_3$, labeled according to their Euclidean order.  Then, for each $i\in \{1,2,3\}$, $I_i$ has a similar decomposition $I_1=I_{i1}\cup I_{i2}\cup I_{i3}$.  Continuing inductively,  we decompose $I$ into sub-intervals $I_{\omega},  \omega\in \{1,2,3\}^m$.  For each $\omega\in \{1,2,3\}^m$, let us choose an arbitrary point $x_{\omega}\in I_{\omega}$.  For each $\omega\in \{1,2,3\}^{m-1}$ we have 
\[
\frac{P(x_{\omega*1})-P(x_{\omega*3})}{x_{\omega*1}-x_{\omega*3}}=P'(z_{\omega})
\]
for some point $z_{\omega}\in \mathcal{I}_{\omega}$.  By  \eqref{eq:lemma-smallvalue2,1},
\[|P'(z_{\omega})|< \frac{2\cdot a^{-(mq)^2}}{(a^{-mq}\cdot 3^{-m})}.\]
Similarly,  we see via the triangle inequality that for each $\omega\in \{1,2,3\}^{m-2}$, there exists $z_{\omega}\in I_{\omega}$ such that 
\[|P^{(2)}(z_{\omega})|< \frac{2^2\cdot a^{-(mq)^2}}{(a^{-mq}\cdot 3^{-m})(a^{-mq}\cdot 3^{-m+1})}.\]
Via $m$ iterations of this process, we can find a point $x_0\in I$ such that 
\begin{equation*}
    |P^{(m)}(x_0)|<\frac{2^{m}\cdot a^{-(mq)^2}}{(a^{-mq})^m\cdot 3^{-m}\cdot3^{-m+1}\cdots 3^{-1}}=a^{-(mq)^2+qm^{2}+m\log_a 2+\frac{m(m+1)}{2}\log_a 3}< 1,
\end{equation*}
where the last inequality follows from the assumptions that 
$q> 2(1+\log_a 2+\log_a 3)$ and $m\geq 1$. 
This is impossible, since $|l_m|\geq 1$ implies
\begin{equation*}
    |P^{(m)}(x)|=|l_m|\cdot m!\geq  1, \; \forall x\in [a, b].
\end{equation*}
It follows that each of the  $2k+2$ intervals that cover $A$ has diameter at most $ a^{-mq}$, as claimed.
\end{proof}

\noindent\textbf{Proof of Lemma \ref{smallvalue1}.} 
Let $\epsilon>0$ be small, and let $\delta=\epsilon\cdot 3^{-(k+1)}$. We say that a $j$-tuple  $(u_1,\ldots, u_j)$ with $2\leq j\leq k$ and $0\leq u_j<\ldots<u_1\leq N$ is \emph{admissible} if 
$$u_i-u_{i+1}\leq N^{3^{i-1}\delta} \text{ for all }1\leq i\leq j-1.$$
Additionally,  $(u_1)$ is \emph{admissible}   by convention.

Now, let $h(x)=\sum_{j=1}^k l_j x^{u_i}\in\mathcal{F}_{k, N}$. Let $j$ be the largest integer such that $(u_1,\cdots,u_j)$ is admissible but $(u_1,\cdots,u_j, u_{j+1})$ is not admissible. Note that it might happen that $j=k$, and in such a case, we just put $u_{j+1}=-\infty$. Thus, we always have 
\begin{equation*}
    u_j-u_{j+1}>N^{3^{j-1}\delta}.
\end{equation*}
Since $(u_1,\cdots,u_j)$ is admissible 
\begin{equation*}
    u_i-u_{i+1}\leq N^{3^{i-1}\delta}, \;\forall 1\leq i\leq j-1.
\end{equation*}
Let
$$T_1(x):=\sum_{i=1}^{j} l_i x^{u_i}\, \text { and } T_2(x):=\sum_{i=j+1}^{k} l_i x^{u_i},\, \text{ so that } h(x)=T_1(x)+T_2(x).$$ 
We rewrite $T_1$ as 
\begin{equation}\label{T1}
   T_1(x)=x^{u_j}\cdot \left(l_1 x^{u_1-u_j}+\ldots+l_{j}\right)=
   x^{u_j-N^{3^{j-1}\delta}}\cdot x^{N^{3^{j-1}\delta}}\cdot \left(l_1 x^{u_1-u_j}+\ldots+l_{j}\right).
\end{equation}
Let $q\in \mathbb{R}$ be such that
\begin{equation}\label{q}
   \left( \left(u_1-u_j \right)q \right)^2+5\log_a N=N^{3^{j-1}\delta}.
\end{equation}
Note that we have 
$$u_1-u_j\le N^{\delta}+\cdots +N^{3^{j-2}\delta}\le 2N^{3^{j-2}\delta}.$$
Therefore, as $N\gg 1$,  
$$g> 2+2\log_a2+2\log_a3.$$

Applying Lemma \ref{smallvalue2} to the polynomial $P(x)=l_1 x^{u_1-u_j}+\ldots+l_{j}$, it follows that there are at most $2k+2$ intervals $I_i\subset [a, b]$ satisfying 
\begin{equation}\label{claim1}
   |I_i|\leq a^{-(u_1-u_j)q}\;\text{for each}\;i,
\end{equation}
and 
\begin{equation}\label{P}
    |P(x)|\geq a^{-(u_1-u_j)^2q^2}\;\text{for all}\;x\in [a, b]\setminus\cup_i I_i.
\end{equation}
By \eqref{T1}, \eqref{q} and \eqref{P}, we have 
\begin{align*}
    |T_1(x)|\geq & x^{u_j-N^{3^{j-1}\delta}}a^{N^{3^{j-1}\delta}}a^{-(u_1-u_j)^2q^2}\\
    =& x^{u_j-N^{3^{j-1}\delta}}a^{4\log_a N} \;\text{for all}\; x\in [a, b]\setminus\cup_i I_i.
\end{align*}
Since $|l_{j+1}|, \ldots, |l_k|\leq N^4$ and $u_j-u_{j+1}>N^{3^{j-1}\delta}$, for $N\gg k$, we have 
\begin{equation*}
    |T_2(x)|\leq k a^{\log_a N^4}x^{u_i+1}\leq\frac{|T_1(x)|}{2}.
\end{equation*}
It follows that 
\begin{equation*}
    |h(x)|=|T_1(x)+T_2(x)|\geq\frac{1}{2}a^{u_j-N^{3^{j-1}\delta}}a^{5\log_a N} \;\text{for all}\; x\in [a, b]\setminus\cup_i I_i.
\end{equation*}
Recall that $3^{n+1}\delta=\epsilon$, when $N\gg k$, we get
\begin{equation*}
    |h(x)|\geq a^{u_1-N^{\epsilon}} \;\text{for all}\; x\in [a, b]\setminus\cup_i \mathcal{I}_i.
\end{equation*}
This is the second assertion.

Finally, the first assertion follows by combining \eqref{claim1}, and
\begin{equation*}
u_1-u_j\leq 2N^{3^{j-2}\delta},\quad \text{ and }   (u_1-u_j)q=(N^{3^{j-1}\delta}-5\log_a N)^{1/2}
\end{equation*}
from which a suitable  choice of $N^{\sigma}$ can be made.
\qed

\subsection{Proof of Theorem \ref{main-2}}
Let $\mu=\mu_\mathbf{p}$ be our non-atomic self-similar measure on $[a,b]$, and recall our other notations from the previous sections. Let $g\in \mathbb{Z}[x]$ be a polynomial as in \eqref{polynomial}, and recall the other parameters and notations from Theorem \ref{main-2}.  We aim to prove the existence of some $\tau>0$ such that
\begin{equation*}
\left|\int e^{2\pi i g(x) }d\mu(x) \right|\le a^{-N^{\tau}}, \text{ where } \tau \text { is independent of  } g.
\end{equation*}

We begin with the following reduction:
\begin{Lemma}  \label{Lemma relation}
If \eqref{eq:main-thm-estimate-1} holds whenever $1<a<b<a^2$, then it holds for arbitrary $1<a<b$.
\end{Lemma}
\begin{proof}
Since $\Phi$ is uniformly contracting, there exists some $N$ such that: for every $\omega \in \mathcal{A}^N$ there exists $1<a(\omega)<b(\omega)<\left( a(\omega) \right)^2$ that satisfy
$$\supp \left( f_\omega \mu \right) \subseteq [a(\omega),\,b(\omega)].$$
It follows that the self-similar measure $f_\omega \mu$ satisfies \eqref{eq:main-thm-estimate-1}, possibly with $\tau=\tau(\omega)$  varying across the different $\omega \in \mathcal{A}^N$. By self-similarity, for all  $g$ as in Theorem \ref{main-2},
$$\left| \int e^{2\pi i g(x) }d\mu(x) \right|$$
is bounded above by a convex combination over
$$\left \lbrace \left| \int e^{2\pi i g(x) }\,df_\omega \mu(x) \right|: \omega \in \mathcal{A}^N \right \rbrace.$$
So, Theorem \ref{main-2} holds for $\mu$ (with some choice of $\tau>0$).
\end{proof}
By Lemma \ref{Lemma relation}, we may  assume that there exists some $0<\delta_0<1$ such that
\begin{equation} \label{eq:assumption-small-interval-support-1}
\frac{b}{a^{ 2\delta_0 }}<1.
\end{equation}
We fix a parameter $q=q(N)$ to be chosen later, and roughly follow the proof outline of Theorem \ref{thm:van_der_Corput_C1+a}:
$$ $$

\noindent{\textbf{Step 1: Linearization and choice of $q$.}} 
Recall the definition of the cut-set \eqref{Def W k}
\begin{equation*} 
\mathcal{W}_q  = \lbrace \omega \in \mathcal{A}^*: |r_\omega| \leq q < |r_{\omega^{-}}| \rbrace.
\end{equation*}
We have the following linearization lemma:
\begin{Lemma} \label{Lemma Lin poly}
We have 
\begin{equation} \label{eq:linearization-argument-1}
\left| \int e^{2\pi i g(x) }d\mu(x) \right|\leq \sum_{\omega\in \mathcal{W}_q}p_\omega \left| \widehat{\mu} \left( g'(t_\omega)\cdot r_\omega\right) \right|+O\left(N^2\cdot u_1 ^2 \cdot b^{u_1} \cdot q^2 \right).
\end{equation}
\end{Lemma}
\begin{proof}
By (a minor variation of) Lemma \ref{Lemma basic}, as $\Phi$ is a self-similar IFS
\begin{equation} \label{eq:elementary-decomp-1}
\int e^{2\pi i g(x) }d\mu(x) = \sum_{\omega\in \mathcal{W}_q}p_\omega \int e^{2\pi i g(f_\omega (x))}d\mu(x).
\end{equation}
Recall that 
\begin{equation*}
    g(x)=\sum_{j=1}^{k-1}l_j (x^{u_j}-x^{u_{j+1}})+\sum_{j=1}^{k-1}m_j (x^{v_j}-x^{v_{j+1}}),
\end{equation*}
with $l_j, m_j\in\mathbb{Z}, 0<|l_j|, |m_j|\leq N^{1+\epsilon}$ and $1\leq u_k<\ldots<u_1\leq N, 1\leq v_k<\ldots<v_1\leq N$, where $k\geq 2, N\gg 1, 0<\epsilon\ll 1$ are fixed and the coefficient of the leading term in $g(x)$ is non-zero. 
 Note that $g(x)$ consists of at most $2k$ distinct terms. Moreover, we have $g(x), g'(x), g''(x)\in\mathcal{F}_{2k, N}$. Using the Taylor expansion and the assumption on $g$, for every $\omega \in \mathcal{A}^*$ we have 
\begin{align}\label{taylor}
g(f_\omega(x))=& g(f_\omega(0))+g'(f_\omega(0))(f_\omega(x)-f_\omega(0))+O(||g''||_\infty)(f_\omega(x)-f_\omega(0))^2)\notag \\
=& g(t_\omega)+g'(t_\omega)\cdot r_\omega\cdot x+O(N^{2}u_1^2b^{u_1} r_\omega^2).
\end{align}
The lemma follows by combining \eqref{eq:elementary-decomp-1}, \eqref{taylor},  the Lipschitz property of the complex exponent, and the fact that $r_\omega \leq q$ for $\omega \in \mathcal{W}_q$, similarly to the first step in the proof of Theorem \ref{thm:van_der_Corput_C1+a}.
\end{proof}
We now choose
\begin{equation} \label{choice of q}
q= \left( N^2 a^{u_1} \right)^{-\delta_0}.
\end{equation}
So, by \eqref{eq:assumption-small-interval-support-1} we can find some $\delta_1>0$ such that
$$ N^2 u_1 ^2 b^{u_1} q^2 = O\left( a^{-u_1 \delta_1} \right).$$
Thus, as a corollary of Lemma \ref{Lemma Lin poly},
\begin{equation} \label{Eq lin after choice}
\left| \int e^{2\pi i g(x) }d\mu(x) \right|\leq \sum_{\omega\in \mathcal{W}_q}p_\omega \left| \widehat{\mu} \left( g'(t_\omega)\cdot r_\omega\right) \right|+O\left( a^{-u_1 \delta_1}  \right).
\end{equation}
$$ $$

\noindent{\textbf{Step 2: Applying Tsujii's Theorem and the covering lemma (Lemma \ref{smallvalue1}).}} Recall  definitions \eqref{Def W k} and \eqref{Def ck}. Fix $r\in \mathcal{C}_q$ and define $\mathcal{W}_{q,r} $ as in \eqref{Eq c k r}, 
\begin{equation*} 
\mathcal{W}_{q,r} = \lbrace \omega \in \mathcal{W}_q:\, r_{\omega} = r \rbrace.
\end{equation*}

Next, fix some $\beta>0$ to be determined later. Applying the Corollary \eqref{Eq. Tsuji} of Tsujii's Theorem \ref{Tsuji original} with this $\beta$, there exists some $c_0=c_0(\beta)>0$ such that for all $b$ large enough, putting $t=\log_a \left(2 u_1 \delta_0 \right)$ such that it grows to $\infty$ by the assumption $u_1\geq N^{\frac{1}{2k}}$ as $N\gg 1$,  we have  
\begin{equation}\label{eq:Tsujii-large-deviation-2}
\left|\left\{
n\in\mathbb{Z}: \exists \lambda\in [n,n+1]\cap[-a^{2\delta_0 u_1}, a^{2\delta_0 u_1}]\;\textrm{s.t.} \;|\widehat{\mu}(\lambda)|\ge a^{-c_0 u_1}
\right\}\right| \le  a^{\beta u_1}.
\end{equation}

We now deal, as before, with the $\mathbb{P}$ mass of the bad frequencies (we use the notation as in Lemma \ref{Lemma prob of bad}). This is where we use our work from Section \ref{Section quan}:  
Applying Lemma \ref{smallvalue1} to $g'(x)$ and $g''(x)$, we know that there exist $\sigma>0$ and at most $8k+8$ intervals $I_i\subset [a, b]$ with $|I_i|\leq a^{-N^{\sigma}}$ such that 
\begin{equation}\label{large-2derivative}
    |g'(x)|\geq a^{u_1-N^{\epsilon}} \; \text{and}\;|g''(x)|\geq a^{u_1-N^{\epsilon}}\; \text{for all}\;x\in [a, b]\setminus \cup_i I_i.
\end{equation}
Let us now define:
\begin{equation} \label{Eq L k r}
\mathcal{L}_{q,r} = \mathcal{W}_{q,r} \setminus \lbrace \omega \in \mathcal{W}_{q,r}:\, t_\omega \in      \bigcup_{i=1} ^{8k+8} I_i   \rbrace.
\end{equation}

Recall that $s_0>0$ is the Frostman exponent of the measure $\mu$  as in Proposition \ref{Lemma Feng}. 
\begin{Lemma} \label{Lemma prob of bad 2}
For every fixed $n\in \mathbb{Z}$,
$$\mathbb{P}\left( [\omega]:\, \omega\in \mathcal{L}_{k,r}, g'(t_{\omega})\cdot r \in [n,n+1] \right) = O\left(  \left( a^{-u_1(1-\delta_0)+N^{\epsilon}}\cdot N^{2\delta_0} +r\right)^{s_0} \right).$$
\end{Lemma}
\begin{proof}
First, we claim that for any connected component $J$ of $[a,b]\setminus \left( \bigcup_{i=1} ^{8k+8} I_i \right)$, the set 
$$\left \lbrace t_\omega:\, \omega  \in \mathcal{L}_{q,r},\, t_\omega \in J, \text{ and } g'(t_\omega) r \in [n,\,n+1] \right \rbrace   $$
is  contained in an interval of length at most $\frac{2}{r_{\min}} \cdot a^{-u_1(1-\delta_0)+N^{\epsilon}}\cdot N^{2\delta_0}$. 

Indeed, 
it follows from \eqref{large-2derivative} that for $\omega,\eta\in  \mathcal{L}_{k,r}$, if
$$|t_\omega-t_\eta|> c\, \text{ for some }c>0, \text{ and } [t_\omega, t_\eta]\cap\cup_i I_i=\emptyset \text{ assuming } t_\omega<t_\eta,$$ 
then 
\begin{equation}\label{eq:expanding-property-1}
|g'(t_\omega) r-g'(t_\eta) r|\geq  a^{u_1-N^{\epsilon}}c|r|.
\end{equation}
So, if $c= \frac{2}{r_{\min}} \cdot  a^{-u_1(1-\delta_0)+N^{\epsilon}}\cdot N^{2\delta_0}$, we see that, by \eqref{choice of q} and since  $|r|\geq q\cdot r_{\min}$
\begin{equation}\label{eq:expanding-property-2}
|g'(t_\omega) r-g'(t_\eta) r|\geq  2.
\end{equation}
This is impossible since $\diam \left( [n,\,n+1]\right)=1$ and $g'(t_\omega) r,\,g'(t_\eta) r\in [n,n+1]$ by our assumptions.

Thus, invoking Proposition \ref{Lemma Feng} and arguing similarly as in Lemma \ref{Lemma prob of bad}, we obtain 
$$\mathbb{P}\left( [\omega]:\, \omega\in \mathcal{L}_{k,r},g'(t_{\omega})\cdot r \in [n,n+1] \right) = O \left( \left( a^{-u_1(1-\delta_0)+N^{\epsilon}}\cdot N^{2\delta_0} +r\right)^{s_0} \right). $$
\end{proof}

$$ $$

\noindent{\textbf{Step 3: Collecting the error terms and conclusion of the proof.}} Recalling \eqref{Eq L k r} and the term $c_0>0$ from \eqref{eq:Tsujii-large-deviation-2},  we partition $\mathcal{W}_{k,r}$ into three subsets:
\begin{itemize}
\item $\mathcal{W}_{q,r,1}$ is the collection of all $\omega \in \mathcal{L}_{q,r} $ such that
$$ \left| \widehat{\mu} \left( g'(t_\omega)\cdot r  \right) \right|\geq a^{-c_0\cdot u_1}.$$

\item $\mathcal{W}_{k,r,2}$ is the collection of all the other $\omega \in \mathcal{L}_{q,r}$, that is,  
$$\mathcal{L}_{k,r,2} = \mathcal{L}_{q,r} \setminus \mathcal{L}_{q,r,1}.$$ 

\item $\mathcal{W}_{q,r,3}$ is the family of all the other $\omega \in \mathcal{W}_{q,r}$,
$$\mathcal{W}_{q,r,3} = \mathcal{W}_{q,r} \setminus \mathcal{L}_{q,r}.$$
\end{itemize}

Now, for any $\omega\in  \mathcal{W}_{q,r}$ and $N$ large enough,
\begin{equation*}
    |g'(t_I) r_I|\leq 4k N^{1+\epsilon}b^{u_1}\leq a^{2\delta_0 u_1}.
\end{equation*}
This is important, since we are now in a position to apply \eqref{eq:Tsujii-large-deviation-2} to our $\mathcal{W}_{q,r,1}$, as we do below. Note  that, due to the properties of the coding map $\pi$ (as in e.g. Lemma \ref{Lemma prob of bad}), $\{\pi([\omega])\}_{\omega\in \mathcal{W}_{q,r,3}}$ is contained in at most $8k+8$ intervals of length less than $a^{-N^{\sigma}}+r$. Applying  \eqref{Eq lin after choice},  and omitting global multiplicative constants from notation, using Lemma \ref{Lemma basic} to estimate the size of $\mathcal{C}_q$, we get
\begin{align*}
&\left|\int e^{2\pi i g(x) }d\mu(x)\right| \leq  \sum_{\omega\in \mathcal{W}_b}p_\omega \left| \widehat{\mu} \left( g'(t_\omega)\cdot r_\omega\right) \right|+a^{-u_1 \delta_1} \\
\le & \sum_{r\in \mathcal{C}_q} \left( \left|\sum_{\omega\in \mathcal{W}_{q,r,1}}p_\omega\widehat{\mu}\left(g'(t_\omega) r\right)\right| +\left|\sum_{\omega \in \mathcal{W}_{q,r,2}}p_\omega\widehat{\mu}\left(g'(t_\omega) r \right)\right|+\left|\sum_{\omega \in \mathcal{W}_{q,r,3}}p_\omega\widehat{\mu}\left(g'(t_\omega) r\right)\right|\right)+a^{-u_1 \delta_1}\\
\le  & \sum_{r\in \mathcal{C}_q} \left( \sum_{n\in \Z}\left(\sum_{\substack{\\ \omega\in \mathcal{W}_{q,r,1} \\ g'(t_\omega)r_\omega \in \left[n, n+1\right] }}p_\omega\right)+
\sum_{\omega\in \mathcal{W}_{q,r,2}}p_\omega a^{-c_0 u_1}+\sum_{\omega\in \mathcal{W}_{q,r,3}}p_\omega \right)+a^{-u_1 \delta_1}\\
\le & \log(q^{-1})\cdot \left(  a^{\beta u_1}\cdot \left(a^{-u_1(1-\delta_0)+N^{\epsilon}}\cdot N^{2\delta_0} +\left( N^2 a^{u_1} \right)^{-\delta_0}\right)^{s_0} +a^{-c_0 u_1}+(8k+8)\cdot \left( a^{-N^{\sigma}}+\left( N^2 a^{u_1} \right)^{-\delta_0}\right)^{s_0}\right)+a^{-u_1 \delta_1}.
\end{align*}
Putting $\beta < \min \lbrace \delta_0, \, 1-\delta_0 \rbrace$ and recalling that  $u_1\geq N^{\frac{1}{2k}}$ and $\epsilon \ll 1$, we get 
\[
\left|\int e^{2\pi i g(x) }d\mu(x)\right|\leq a^{-N^{\tau}}
\] 
for some constant $\tau$ which is independent of the polynomial $g$. The proof is complete. \hfill{$\Box$}

\subsection{Proof of Theorem \ref{Main Theorem pos}}
 Theorem \ref{Main Theorem pos} will be deduced from Theorem \ref{main-2} via the following Proposition. It works in significantly greater generality, and is due to Technau and Yesha \cite{Yesha2023Technau}. Recall the definitions and the notations as in Section \ref{Section corr} and Theorem \ref{Main Theorem pos}. In particular, $\mu$ is our non-atomic self-similar measure. 
 \begin{Proposition} \cite[Proposition 7.1]{Yesha2023Technau} \label{Prop yesha}
Let $k\geq 2$ and let $C_k(N)$  be a sequence such that $\lim_{N\rightarrow \infty} C_k (N)=1$. Suppose that there exists some $\tau>0$ such that for all $f\in C_c ^\infty \left(\mathbb{R}^{k-1} \right)$ we have
$$\int_J \left( R_k(f, (x^n \mod 1)_{n\geq 1}, N)-C_k(N)\cdot \int_{\mathbb{R}^{k-1}}f(\textbf{x})  d\textbf{x} \right)^2 d\mu (x) = O\left( N^{-\tau} \right),\, \text{ as } N\rightarrow \infty,$$
where $\mu$ is a supported on a bounded interval $J$. 

Then for $\mu$-a.e.~$x$ the sequence $(x^n \mod 1)_{n\geq 1}$ has Poissonian $k$-point correlations.
 \end{Proposition}
We note that Proposition \ref{Prop yesha} is stated and proved in \cite[Proposition 7.1]{Yesha2023Technau} for the Lebesgue measure on $J$, but it is not hard to see that the argument works for any Borel probability measure on $J$.
$$ $$

\noindent{\textbf{Proof of Theorem \ref{Main Theorem pos}}
   With no loss of generality, we assume $\xi=1$, and the other cases are similar. For any $f\in C^{\infty}_c(\mathbb{R}^{k-1})$ let $\widehat{f}$ denote its Fourier transform, and, recalling the definitions from Section \ref{Section corr}, put
$$C_k(N):=\frac{\sharp\mathcal{U}_k}{N^k}=\left(1-\frac{1}{N}\right)\cdots\left(1-\frac{k-1}{N}\right).$$
Let $\epsilon>0$ be an auxiliary parameter to be decided later. By Poisson summation we have
   \begin{align*}
    R_k(f, \{x^n\})_{n\geq 1}, N)
    =&\frac{1}{N^k}\sum_{\textbf{l}\in\mathbb{Z}^{k-1}}\widehat{f}\left(\frac{\textbf{l}}{N}\right)\sum_{\textbf{u}\in\mathcal{U}_k}
    e^{-2\pi i \sum_{j=1}^{k-1}l_j (x^{u_j}-x^{u_{j+1}})} \notag \\ 
    =&C_k(N)\widehat{f}(0)+\frac{1}{N^k}\sum_{\textbf{l}\in\mathbb{Z}^{k-1}\setminus\{\textbf{0}\}}\widehat{f}\left(\frac{\textbf{l}}{N}\right)\sum_{\textbf{u}\in\mathcal{U}_k}
    e^{-2\pi i \sum_{j=1}^{k-1}l_j (x^{u_j}-x^{u_{j+1}})}\\
    =&C_k(N)\widehat{f}(0)+\frac{1}{N^k}\sum_{ \textbf{l}\in\mathbb{Z}^{k-1}\atop 0<\|\textbf{l}\|_{\infty}\leq N^{1+\epsilon}}\widehat{f}\left(\frac{\textbf{l}}{N}\right)\sum_{\textbf{u}\in\mathcal{U}_k}
    e^{-2\pi i \sum_{j=1}^{k-1}l_j (x^{u_j}-x^{u_{j+1}})}\\
    &\quad\quad\quad +\frac{1}{N^k}\sum_{\textbf{l}\in\mathbb{Z}^{k-1}\atop\|\textbf{l}\|_{\infty}>N^{1+\epsilon}}\widehat{f}\left(\frac{\textbf{l}}{N}\right)\sum_{\textbf{u}\in\mathcal{U}_k}
    e^{-2\pi i \sum_{j=1}^{k-1}l_j (x^{u_j}-x^{u_{j+1}})},\\
\end{align*}
where $0<\|\textbf{l}\|_{\infty}\leq N^{1+\epsilon}$ means $0<|l_j|\leq N^{1+\epsilon}$ for each $j$. We rewrite the last displayed equation as
$$C_k(N)\cdot \widehat{f}(0)+X_N(x)+Y_N(x).$$
Since $f\in C^\infty(\mathbb{R})$, $\widehat{f}$ decays faster than any polynomial, so there exists $\alpha>0$ such that $|Y_N(x)|=O(N^{-\alpha})$ uniformly in  $x\in\spt(\mu)$. Next, we have 
$$\widehat{f}(0)=\int_{\mathbb{R}^{k-1}}f(\textbf{x})  d\textbf{x}, \text{ and } \lim_{N\rightarrow \infty}C_k(N)=1.$$ Thus, applying Proposition \ref{Prop yesha}, it remains to show that $\lim_{N\rightarrow\infty}X_N=0$ in $L^2 (\mu)$ with a polynomial rate of convergence. This is proved below.

Let 
 $$\mathcal{\tilde{U}}_k=\{\textbf{u}=(u_1,\ldots,u_k): u_i\in \lbrace 1,...,N\rbrace,\, 1\leq u_k<\ldots<u_1\leq N\}.$$ 
By symmetry we have 
\begin{align}\label{L2bound}
  &\quad\int |X_N(x)|^2 d\mu(x)\notag\\
 &=\int \Big|\frac{1}{N^k}\sum_{0<\|\textbf{l}\|_{\infty}\leq N^{1+\epsilon}}\widehat{f}\left(\frac{\textbf{l}}{N}\right)\sum_{\textbf{u}\in\mathcal{U}_k}
    e^{-2\pi i \sum_{j=1}^{k-1}l_j (x^{u_j}-x^{u_{j+1}})}\Big|^2 d\mu(x)\notag\\
 &=\frac{1}{N^{2k}}\sum_{0<\|\textbf{l}\|_{\infty}, \|\textbf{m}\|_{\infty}\leq N^{1+\epsilon}}\widehat{f}\left(\frac{\textbf{l}}{N}\right)
 \widehat{f}\left(\frac{\textbf{m}}{N}\right)\sum_{\textbf{u}, \textbf{v}\in\mathcal{U}_k} \int e^{2\pi i g_{\textbf{l}, \textbf{m}, \textbf{u}, \textbf{v}}(x)}d\mu(x)\notag\\
 &\le\frac{2(k!)^2}{N^{2k}}\sum_{0<\|\textbf{l}\|_{\infty}, \|\textbf{m}\|_{\infty}\leq N^{1+\epsilon}}\widehat{f}\left(\frac{\textbf{l}}{N}\right)
 \widehat{f}\left(\frac{\textbf{m}}{N}\right)\sum_{u_1\geq v_1, \textbf{u}, \textbf{v}\in\mathcal{\tilde{U}}_k} \int e^{2\pi i g_{\textbf{l}, \textbf{m}, \textbf{u}, \textbf{v}}(x)}d\mu(x),
\end{align}
where 
$$g_{\textbf{l}, \textbf{m}, \textbf{u}, \textbf{v}} (x)=:\sum_{j=1}^{k-1}l_j (x^{u_j}-x^{u_{j+1}})+\sum_{j=1}^{k-1}m_j (x^{v_j}-x^{v_{j+1}}).$$
We decompose the region of of parameters 
$$\Omega=\left\{(\textbf{l}, \textbf{m}, \textbf{u}, \textbf{v}): 0<\|\textbf{l}\|_{\infty}, \|\textbf{m}\|_{\infty}\leq N^{1+\epsilon}, u_1\geq v_1, \textbf{u}, \textbf{v}\in\mathcal{\tilde{U}}_k\right\}$$ into three parts: 
\begin{align*}
    \Omega_1= &\left\{(\textbf{l}, \textbf{m}, \textbf{u}, \textbf{v})\in\Omega: u_1\leq N^{\frac{1}{2k}}\right\},\\
    \Omega_2= &\left\{(\textbf{l}, \textbf{m}, \textbf{u}, \textbf{v})\in\Omega: u_1>N^{\frac{1}{2k}}, u_1=v_1, l_1+m_1=0\right\},\\
    \Omega_3= &\Omega\setminus \cup_{i=1}^2\Omega_i.
\end{align*}
 Then
\begin{equation}\label{region1}
 \frac{1}{N^{2k}}   \left|\sum_{(\textbf{l}, \textbf{m}, \textbf{u}, \textbf{v})\in\Omega_1}\widehat{f}\left(\frac{\textbf{l}}{N}\right)
 \widehat{f}\left(\frac{\textbf{m}}{N}\right)\int e^{2\pi i g_{\textbf{l}, \textbf{m}, \textbf{u}, \textbf{v}}(x)}d\mu(x)\right|\leq \frac{1}{N^{2k}}\cdot N^{2(k-1)(1+\epsilon)}\cdot (N^{\frac{1}{2k}})^{2k}=\frac{1}{N^{1-2k\epsilon}}.
\end{equation}
By taking $\epsilon<\frac{1}{2k}$ we  see that the term corresponding to summation over $\Omega_1$ decays polynomially. Next, by Theorem \ref{main-2},  
\begin{equation}\label{region3}
  \frac{1}{N^{2k}}  \left|\sum_{(\textbf{l}, \textbf{m}, \textbf{u}, \textbf{v})\in\Omega_3}\widehat{f}\left(\frac{\textbf{l}}{N}\right)
 \widehat{f}\left(\frac{\textbf{m}}{N}\right)\int e^{2\pi i g_{\textbf{l}, \textbf{m}, \textbf{u}, \textbf{v}}(x)}d\mu(x)\right|\leq \frac{1}{N^{2k}}\cdot N^{2(k-1)(1+\epsilon)}\cdot N^{2k}\cdot a^{-N^{\tau}}=O\left(a^{-N^{\tau/2}}\right).
\end{equation}
For the term corresponding to summation over $\Omega_2$, we can compare the values of $u_2, v_2$ and proceed similarly as in the estimates \eqref{region1}
and \eqref{region3} to conclude that 
\begin{equation}\label{region2}
 \frac{1}{N^{2k}}   \left|\sum_{(\textbf{l}, \textbf{m}, \textbf{u}, \textbf{v})\in\Omega_2}\widehat{f}\left(\frac{\textbf{l}}{N}\right)
 \widehat{f}\left(\frac{\textbf{m}}{N}\right)\int e^{2\pi i g_{\textbf{l}, \textbf{m}, \textbf{u}, \textbf{v}}(x)}d\mu(x)\right|= O\left(N^{-c}\right), \text{ for some } c\in (0,1).
\end{equation}
Plugging \eqref{region1}-\eqref{region2} into \eqref{L2bound}, we see that $\int |X_N(x)|^2 d\mu(x)$ decays at a polynomial speed. Theorem \ref{Main Theorem pos} now follows directly from Proposition \ref{Prop yesha}. \hfill{$\Box$}

\bibliography{bib}

\begin{thebibliography}{10}

\bibitem{Aistleitner2021Baker}
Christoph Aistleitner and Simon Baker.
\newblock On the pair correlations of powers of real numbers.
\newblock {\em Israel J. Math.}, 242(1):243--268, 2021.

\bibitem{Yesha2023Baker}
Christoph Aistleitner, Simon Baker, Niclas Technau, and Nadav Yesha.
\newblock Gap statistics and higher correlations for geometric progressions
  modulo one.
\newblock {\em Math. Ann.}, 385(1-2):845--861, 2023.

\bibitem{Chris2018Lach}
Christoph Aistleitner, Thomas Lachmann, and Florian Pausinger.
\newblock Pair correlations and equidistribution.
\newblock {\em J. Number Theory}, 182:206--220, 2018.

\bibitem{Algom2022Baker}
Amir Algom, Simon Baker, and Pablo Shmerkin.
\newblock On normal numbers and self-similar measures.
\newblock {\em Adv. Math.}, 399:Paper No. 108276, 17, 2022.

\bibitem{algom2021decay}
Amir Algom, Federico~Rodriguez Hertz, and Zhiren Wang.
\newblock Logarithmic {F}ourier decay for self conformal measures.
\newblock {\em J. Lond. Math. Soc. (2)}, 106(2):1628--1661, 2022.

\bibitem{algom2023polynomial}
Amir Algom, Federico~Rodriguez Hertz, and Zhiren Wang.
\newblock Polynomial {F}ourier decay and a cocycle version of {D}olgopyat's
  method for self conformal measures.
\newblock {\em arXiv preprint arXiv:2306.01275}, 2023.

\bibitem{algom2020decay}
Amir Algom, Federico Rodriguez~Hertz, and Zhiren Wang.
\newblock Pointwise normality and {F}ourier decay for self-conformal measures.
\newblock {\em Adv. Math.}, 393:Paper No. 108096, 72, 2021.

\bibitem{Avila2010jairo}
Artur Avila, Jairo Bochi, and Jean-Christophe Yoccoz.
\newblock Uniformly hyperbolic finite-valued
  {$\text{SL}(2,\mathbb{R})$}-cocycles.
\newblock {\em Comment. Math. Helv.}, 85(4):813--884, 2010.

\bibitem{bakwer2021equi}
Simon Baker.
\newblock Equidistribution results for self-similar measures.
\newblock {\em Int. Math. Res. Not. IMRN}, (16):12378--12401, 2022.

\bibitem{Baker2023Banaji}
Simon Baker and Amlan Banaji.
\newblock Polynomial {F}ourier decay for fractal measures and their
  pushforwards.
\newblock {\em preprint}, 2023.

\bibitem{baker2023spectral}
Simon Baker and Tuomas Sahlsten.
\newblock Spectral gaps and {F}ourier dimension for self-conformal sets with
  overlaps.
\newblock {\em arXiv preprint arXiv:2306.01389}, 2023.

\bibitem{Bour2017dya}
Jean Bourgain and Semyon Dyatlov.
\newblock {F}ourier dimension and spectral gaps for hyperbolic surfaces.
\newblock {\em Geom. Funct. Anal.}, 27(4):744--771, 2017.

\bibitem{bremont2019rajchman}
Julien Br\'{e}mont.
\newblock Self-similar measures and the {R}ajchman property.
\newblock {\em Ann. H. Lebesgue}, 4:973--1004, 2021.

\bibitem{Buf2014Sol}
Alexander~I. Bufetov and Boris Solomyak.
\newblock On the modulus of continuity for spectral measures in substitution
  dynamics.
\newblock {\em Adv. Math.}, 260:84--129, 2014.

\bibitem{Cucker1999Smale}
Felipe Cucker, Pascal Koiran, and Steve Smale.
\newblock A polynomial time algorithm for {D}iophantine equations in one
  variable.
\newblock {\em J. Symbolic Comput.}, 27(1):21--29, 1999.

\bibitem{Dai2012ber}
Xin-Rong Dai.
\newblock When does a {B}ernoulli convolution admit a spectrum?
\newblock {\em Adv. Math.}, 231(3-4):1681--1693, 2012.

\bibitem{Dai2007Feng}
Xin-Rong Dai, De-Jun Feng, and Yang Wang.
\newblock Refinable functions with non-integer dilations.
\newblock {\em J. Funct. Anal.}, 250(1):1--20, 2007.

\bibitem{Melo1993strein}
Welington de~Melo and Sebastian van Strien.
\newblock {\em One-dimensional dynamics}, volume~25 of {\em Ergebnisse der
  Mathematik und ihrer Grenzgebiete (3) [Results in Mathematics and Related
  Areas (3)]}.
\newblock Springer-Verlag, Berlin, 1993.

\bibitem{Feng2009Lau}
De-Jun Feng and Ka-Sing Lau.
\newblock Multifractal formalism for self-similar measures with weak separation
  condition.
\newblock {\em J. Math. Pures Appl. (9)}, 92(4):407--428, 2009.

\bibitem{Larcher2017equi}
Sigrid Grepstad and Gerhard Larcher.
\newblock On pair correlation and discrepancy.
\newblock {\em Arch. Math. (Basel)}, 109(2):143--149, 2017.

\bibitem{Sahl2016Jor}
Thomas Jordan and Tuomas Sahlsten.
\newblock {F}ourier transforms of {G}ibbs measures for the {G}auss map.
\newblock {\em Math. Ann.}, 364(3-4):983--1023, 2016.

\bibitem{Kaufman1984ber}
Robert Kaufman.
\newblock On {B}ernoulli convolutions.
\newblock In {\em Conference in modern analysis and probability ({N}ew {H}aven,
  {C}onn., 1982)}, volume~26 of {\em Contemp. Math.}, pages 217--222. Amer.
  Math. Soc., Providence, RI, 1984.

\bibitem{khalil2023exponential}
Osama Khalil.
\newblock Exponential mixing via additive combinatorics.
\newblock {\em arXiv preprint arXiv:2305.00527}, 2023.

\bibitem{Rudnick1999par}
P\"{a}r Kurlberg and Ze\'{e}v Rudnick.
\newblock The distribution of spacings between quadratic residues.
\newblock {\em Duke Math. J.}, 100(2):211--242, 1999.

\bibitem{Li2018decay}
Jialun Li.
\newblock Decrease of {F}ourier coefficients of stationary measures.
\newblock {\em Math. Ann.}, 372(3-4):1189--1238, 2018.

\bibitem{li2018fourier}
Jialun Li.
\newblock {F}ourier decay, renewal theorem and spectral gaps for random walks
  on split semisimple {L}ie groups.
\newblock {\em Ann. Sci. \'{E}c. Norm. Sup\'{e}r. (4)}, 55(6):1613--1686, 2022.

\bibitem{Li2021Naud}
Jialun Li, Fr\'{e}d\'{e}ric Naud, and Wenyu Pan.
\newblock Kleinian {S}chottky groups, {P}atterson-{S}ullivan measures, and
  {F}ourier decay.
\newblock {\em Duke Math. J.}, 170(4):775--825, 2021.
\newblock With an appendix by Li.

\bibitem{li2019trigonometric}
Jialun Li and Tuomas Sahlsten.
\newblock Trigonometric series and self-similar sets.
\newblock {\em J. Eur. Math. Soc. (JEMS)}, 24(1):341--368, 2022.

\bibitem{Jens2020equi}
Jens Marklof.
\newblock Pair correlation and equidistribution on manifolds.
\newblock {\em Monatsh. Math.}, 191(2):279--294, 2020.

\bibitem{Mattila2015Fourier}
Pertti Mattila.
\newblock {\em {F}ourier analysis and {H}ausdorff dimension}, volume 150 of
  {\em Cambridge Studies in Advanced Mathematics}.
\newblock Cambridge University Press, Cambridge, 2015.

\bibitem{Shmerkin2018mos}
Carolina~A. Mosquera and Pablo~S. Shmerkin.
\newblock Self-similar measures: asymptotic bounds for the dimension and
  {F}ourier decay of smooth images.
\newblock {\em Ann. Acad. Sci. Fenn. Math.}, 43(2):823--834, 2018.

\bibitem{Naud2005exp}
Fr\'{e}d\'{e}ric Naud.
\newblock Expanding maps on {C}antor sets and analytic continuation of zeta
  functions.
\newblock {\em Ann. Sci. \'{E}cole Norm. Sup. (4)}, 38(1):116--153, 2005.

\bibitem{rapaport2021rajchman}
Ariel Rapaport.
\newblock On the {R}ajchman property for self-similar measures on {$\Bbb R^d$}.
\newblock {\em Adv. Math.}, 403:Paper No. 108375, 53, 2022.

\bibitem{sahlsten2023fourier}
Tuomas Sahlsten.
\newblock {F}ourier transforms and iterated function systems.
\newblock {\em arXiv preprint arXiv:2311.00585}, 2023.

\bibitem{sahlsten2020fourier}
Tuomas Sahlsten and Connor Stevens.
\newblock {F}ourier transform and expanding maps on cantor sets.
\newblock {\em To appear in Amer. J. Math.}

\bibitem{Solomyak2021ssdecay}
Boris Solomyak.
\newblock {F}ourier decay for self-similar measures.
\newblock {\em Proc. Amer. Math. Soc.}, 149(8):3277--3291, 2021.

\bibitem{Stein1993book}
Elias~M. Stein.
\newblock {\em Harmonic analysis: real-variable methods, orthogonality, and
  oscillatory integrals}, volume~43 of {\em Princeton Mathematical Series}.
\newblock Princeton University Press, Princeton, NJ, 1993.
\newblock With the assistance of Timothy S. Murphy, Monographs in Harmonic
  Analysis, III.

\bibitem{streck2023absolute}
Lauritz Streck.
\newblock On absolute continuity and maximal garsia entropy for self-similar
  measures with algebraic contraction ratio.
\newblock {\em arXiv preprint arXiv:2303.07785}, 2023.

\bibitem{Yesha2023Technau}
Niclas Technau and Nadav Yesha.
\newblock On the correlations of {$n^\alpha\bmod1$}.
\newblock {\em J. Eur. Math. Soc. (JEMS)}, 25(10):4123--4154, 2023.

\bibitem{Tsujii2015self}
Masato Tsujii.
\newblock On the {F}ourier transforms of self-similar measures.
\newblock {\em Dyn. Syst.}, 30(4):468--484, 2015.

\bibitem{varju2020fourier}
P\'{e}ter~P. Varj\'{u} and Han Yu.
\newblock {F}ourier decay of self-similar measures and self-similar sets of
  uniqueness.
\newblock {\em Anal. PDE}, 15(3):843--858, 2022.

\bibitem{Yoccoz2004some}
Jean-Christophe Yoccoz.
\newblock Some questions and remarks about {$\text{SL}(2,\mathbb{R})$}
  cocycles.
\newblock In {\em Modern dynamical systems and applications}, pages 447--458.
  Cambridge Univ. Press, Cambridge, 2004.

\end{thebibliography}
\bibliographystyle{plain}

\end{document}